\documentclass{amsart}

\usepackage[utf8]{inputenc}
\usepackage[T1]{fontenc}
\usepackage[english]{babel}
\usepackage{amsthm}
\usepackage{amsmath}
\usepackage{mathtools}
\usepackage{tikz-cd}
\usepackage{amsfonts}
\usepackage{amssymb}
\usepackage{graphicx}
\usepackage{xspace}
\usepackage{enumitem}
\usepackage{float}
\floatstyle{plaintop}
\restylefloat{table}

\usepackage[mathscr]{eucal} 
\usepackage{amsmath} 
\usepackage{epsfig}
\usepackage{amscd}
\usepackage{verbatim}
\usepackage{booktabs}

\newtheorem{THEOR}{Theorem}[section]

\newtheorem{LEM}[THEOR]{Lemma}
\newtheorem{DEF}[THEOR]{Definition}

\newtheorem{REM}[THEOR]{Remark}

\theoremstyle{definition}
\newtheorem{remark}[THEOR]{Remark}
\numberwithin{equation}{section}

\newcommand\Z{\mathbb Z}
\def\ZZ{\Z}
\newcommand\R{\mathbb R} 
\newcommand\Q{\mathbb Q}
\newcommand\C{\mathbb C}
\newcommand{\PP}{{\mathbb P}}
\def\P{\PP}
\newcommand\A{{\mathsf A}}
\newcommand\X{{\mathsf X}}
\newcommand\Ag{\A_g}
\def\ag{\Ag}
\newcommand{\agd}{\ag^\delta}
\newcommand{\M}{{\mathsf{M}}}
\newcommand{\sieg}{\mathfrak{S}}
\newcommand\T{{\mathsf T}}
\newcommand{\zg}{\mathsf{Z}}
\newcommand{\rgb}{\mathsf{R}_{g,b}}
\newcommand{\pgb}{\mathsf{P}_{g,b}}
\newcommand{\braid}{\mathbf{B}_r}

\newcommand{\tG}{{\tilde{G}}}

\newcommand{\tg}{{\tilde{g}}}
\newcommand{\ttheta}{{\tilde{\theta}}}
\newcommand{\tC} {{\tilde{C}}}

\newcommand{\Hom}{\operatorname{Hom}}
\newcommand{\End}{\operatorname{End}}
\newcommand{\Aut}{\operatorname{Aut}}
\newcommand{\GL}{\operatorname{GL}}
\newcommand{\Sp}{\operatorname{Sp}}
\newcommand{\Pic}{\operatorname{Pic}}
\newcommand{\Spec}{\operatorname{Spec}}
\newcommand{\Nm}{\operatorname{Nm}}
\newcommand{\Fix}{\operatorname{Fix}}
\newcommand{\Mod}{\operatorname{Mod}}

\def\i{\mathrm i}

\renewcommand{\phi}{\varphi}
\newcommand{\sx}{\langle}
\newcommand{\xs}{\rangle}
\newcommand{\lra}{\longrightarrow}
\newcommand{\ra}{\rightarrow}
\newcommand{\MAGMA}{\texttt{MAGMA}}

\newcounter{example}[section]
\newenvironment{example}[1][]
{\refstepcounter{example}\par\noindent\textbf{Example~\theexample.}\\}{\medskip}

\begin{document}

	\title{Higher dimensional Shimura varieties in the Prym loci of  ramified double covers}

	 \author[P. Frediani]{Paola Frediani} \address{ Dipartimento di
	Matematica, Universit\`a di Pavia, via Ferrata 5, I-27100 Pavia,
	 Italy } \email{{\tt paola.frediani@unipv.it}}
	
	 \author[G.P. Grosselli]{Gian Paolo Grosselli} \address{Dipartimento di
	 Matematica, Universit\`a di Pavia, via Ferrata 5, I-27100, Pavia,
	  Italy } \email{{\tt g.grosselli@campus.unimib.it}}
	
	 \author[A. Mohajer]{Abolfazl Mohajer} \address{Johannes Gutenberg Universit\"at Mainz, Institut f\"ur Mathematik, Staudingerweg 9, 55099 Mainz, Germany } \email{{\tt mohajer@uni-mainz.de}}

	\thanks{The first two authors are members of GNSAGA of INdAM.
The first two authors were partially supported by national MIUR funds,
PRIN  2017 Moduli and Lie theory,  and by MIUR: Dipartimenti di Eccellenza Program
   (2018-2022) - Dept. of Math. Univ. of Pavia.
    } 

	\begin{abstract}
	In this paper we construct Shimura subvarieties of dimension bigger than one of the moduli space $\A^\delta_{p}$ of polarised abelian varieties of dimension $p$, which are generically contained in the Pym loci of (ramified) double covers. The idea is to adapt the techniques already used to construct Shimura curves in the Prym loci to the higher dimensional case,  namely to use families of Galois covers of $\P^1$. The case of abelian covers is treated in detail, since in this case it is possible to make explicit computations that allow to verify a sufficient condition for such a family to yield a Shimura subvariety of $\A^\delta_{p}$. 
	
	\end{abstract}


	\maketitle

	\section{Introduction}
	The purpose of this paper is to construct Shimura subvarieties of dimension bigger than one of the moduli space of polarised abelian varieties of a given dimension which are generically contained in the Pym loci of (ramified) double covers. The technique is similar to the one used in \cite{cfgp} and \cite{fg}, where Shimura subvarieties of dimension one generically contained in the Prym loci were constructed using Galois covers of ${\mathbb P}^1$. 
	
 Let $ { \rgb}$ be the moduli space of isomorphism classes of triples $[(C, \eta, B)]$ where $C$ is a smooth complex projective curve of genus $g$, $B$ is a reduced effective divisor of degree $b$ on $C$ and $\eta$ is a line bundle on $C$ such that $\eta^2={\mathcal O}_{C}(B)$. This determines a double cover of $C$, $f:\tC\to C$ branched on $B$, with $\tC=\Spec({\mathcal O}_{C}\oplus \eta^{-1})$.
	
	The Prym variety associated to $[(C,\eta, B)]$ is the connected component containing the origin of the kernel of the norm map $\Nm_{f}: J\tC \to JC$.
	If $b>0$, then $\ker \Nm_{f}$ is connected.  It is a  polarised abelian variety of dimension $g-1+\frac{b}{2}$, denoted by $P(C,\eta,B)$ or equivalently by $P(\tC, C)$. 
	
	Denote by $\Xi$ the restriction on $P(\tC,C)$ of the principal polarisation on $J\tC$.
	For $b>0$ the polarisation $\Xi$ is of type $\delta=(1,\dots, 1, \underbrace{2,\dots,2}_{g\text{ times}})$.
	If $b=0$ or $2$ then $\Xi$ is twice a principal polarisation and  $P(\tC, C)$ is endowed with this principal polarisation. Let $\A^\delta_{g-1+\frac{b}{2}}$  be the moduli space of abelian varieties of dimension $g-1+\frac{b}{2} $ with a polarisation of type $\delta$.

	The Prym map $ \pgb: \rgb \lra \A^\delta_{g-1+\frac{b}{2}}$ is the map that associates to a point $[(C,\eta,B)]$ the polarised abelian variety $[(P(C,\eta,B), \Xi)].$		The Prym locus is the closure in $\A^\delta_{g-1+\frac{b}{2}}$ of the image of the map $\pgb$. 
	The map  $\pgb$ is generically finite, if and only if $\dim \rgb \leq \dim {\A}^{\delta}_{g-1+\frac{b}{2}},$
	and this holds if either $b\geq 6$ and $g\geq 1$, or $b=4$ and $g\geq 3$, $b=2$ and $g\geq 5$,  $b=0$ and $g \geq 6$.
	
	If $b =0$  the Prym map is generically injective for $g \geq 7$  (\cite{friedman-smith}, \cite{kanev-global-Torelli}). 
	If $b>0$, in \cite{pietro-vale},  \cite{mn}, \cite{no}, it is  proved the generic injectivity in all the 
	cases except for $b=4$, $g =3$, which was previously studied in \cite{nagarama}, \cite{bcv} and for which 
	the degree of the Prym map is $3$. 
	Recently, a global Prym-Torelli theorem was proved for all $g$ and $b\ge 6$ (\cite{ikeda} for $g=1$ and \cite{no1} for all $g$).
	
In \cite{cfgp} a question about the existence of Shimura subvarieties of $ \A^\delta_{g-1+\frac{b}{2}}$ generically contained in the Prym loci in the cases $b=0,2$ was posed. The expectation is that, similarly to what is expected in the case of the Torelli map, for $g$ sufficiently high, they should not exist. Nevertheless in \cite{cfgp} in the case $b=0,2$, and in \cite{fg} for all $b \geq 0$, examples of Shimura curves generically contained in the Prym loci were exhibited, using Galois covers of the projective line. 

Here we adapt this technique to investigate the existence of higher dimensional Shimura subvarieties in the (possibly ramified) Prym loci and we find several examples. Namely, we consider a family of Galois covers $\tC_t \ra \tC_t/\tG \cong \P^1$, where $\tC_t$ is a smooth projective curve of genus $\tg$, $\tG$ is a group admitting a central involution $\sigma \in \tG$, so that we have a factorisation 
$$\begin{tikzcd}[row sep=tiny]
	\tC_t \arrow{rr}{f_ t} \arrow{rd} & &  \arrow{ld } C_t   = \tC_t /\sx \sigma\xs  \\
	& \P^1 \cong \tC_t/\tG \cong C_t /G&
	\end{tikzcd}$$

	where $G = \tG/\sx \sigma\xs$. 
	
	To such a family of Galois covers we associate the Prym variety $P(\tC_t, C_t)$ of the double cover $f_t: \tC_t \ra C_t$ and we ask whether this family of Pryms $P(\tC_t, C_t)$ yields a Shimura subvariety of  $\A^\delta_{g-1+\frac{b}{2}}$. Here $g$ denotes the genus of $C_t$ and $b$ is the number of ramification points of the map $f_t$. 
	
	In \cite[Theorem~3.2]{fg} (see Theorem \ref{teo1ram}) a sufficient condition is given for a family of abelian covers as above to yield a Shimura subvariety of $\A^\delta_{g-1+\frac{b}{2}}$. This is called condition \eqref{condB} in section 3 and it is the natural generalisation of the analogous sufficient condition used in \cite{fgp}, \cite{M10}, \cite{moonen-oort} to determine Shimura subvarieties generically contained in the Torelli locus. Let us briefly explain  \eqref{condB}. Denote by ${V_t}_{-}$  subspace of $V_t:= H^0(\tC_t, \omega_{\tC_t})$ which is anti-invariant under the action of the involution $\sigma$. Set  $ W_t: = H^0( \tC , \omega^2_{\tC_t} )$ and denote by ${W_t}_+ \subset W_t$ the invariant subspace under the action of $\sigma$.  Consider the multiplication map $ S^2 H^0(\tC, \omega_{\tC_t}) \ra H^0( \tC, \omega_{\tC_t}^2)$. Denote by $m: (S^2{V_t}_{-})^{\tG} \ra {W_t}^{\tG}_+$ the restriction of the multiplication map to the space  $(S^2{V_t}_{-})^{\tG}$ of symmetric tensors in $S^2{V_t}_-$, which are $\tG$-invariant.  Condition \eqref{condB} says that the map $m$ is an isomorphism. 
	
	The same condition was considered in \cite{cfgp} and \cite{fg}, where many examples of one-dimensional families of covers satisfying condition \eqref{condB} where given, thus producing Shimura curves in the Prym loci.  The computations are done with \MAGMA. Condition \eqref{condB} clearly implies that $\dim (S^2{V_t}_{-})^{\tG} = \dim  {W_t}^{\tG}_+$ (condition \eqref{condA} in section 3) and the \MAGMA\ script first verifies this equality. In the case of one dimensional families of covers, once condition \eqref{condA} is verified, condition \eqref{condB} is equivalent to say  that the map $m$ should be non zero. This is clearly easier to verify than to prove that $m$ is an isomorphism when $\dim (S^2{V_t}_{-})^{\tG} = \dim  {W_t}^{\tG}_+ >1$. By the global Prym Torelli theorem proved in \cite{no1}, \cite{ikeda}, if $b \geq 6$, conditions \eqref{condA} and \eqref{condB} are equivalent. This is no longer true in the cases $b \leq 4$. 
	
	Nevertheless also in these cases we are able to construct many examples satisfying condition \eqref{condB}. One technique is to verify  another condition that we denoted by \eqref{condB1} that implies condition \eqref{condB}, and that can be verified analysing the representation of $\tG$ on ${V_t}_-$ (see section 3 for a precise statement). 
	
	This is done by the \MAGMA\ script. 
	
	A class of examples were we are able to verify condition \eqref{condB} is the case of families of covers with abelian Galois group $\tG$. In fact if $\tG$ is abelian, we are able to exhibit a basis for the space  $(S^2{V_t}_{-})^{\tG}$ and to compute the rank of the multiplication map $m$. Thus we have an explicit way to verify condition \eqref{condB}. This is explained in section 4, and it is an application of results on abelian covers obtained in \cite{MZ}, \cite{moh}, \cite{W}.   We apply this technique in several examples, some of which are described in section 5. 
	
	In the case of families of non-abelian covers that do not satisfy condition \eqref{condB1}, we need to adopt ad-hoc techniques, one of which is to use Theorem 3.8 of \cite{fgs} (see e.g. Example 4 in section 5). This result is also used to give an example satisfying  \eqref{condA} but not  \eqref{condB} an showing that in fact this family of covers yields a subvariety of ${\mathsf{A}}_4^{\delta}$ that is not totally geodesic, hence it is not Shimura (Example 5 in section 5). 
	
	All the examples that we found satisfying condition \eqref{condB} are listed in the table in the Appendix. Here we summarise the results: 
	
	\begin{THEOR}
	There are 93 families of Galois covers yielding Shimura subvarieties of  $\A^\delta_{p}$ of dimension $n:=r-3$, $2 \leq n \leq 6$, generically contained in the (possibly ramified) Prym loci. The highest value of $p$ is 16, obtained by a 2-dimensional familiy with $\tg=29$, $g=13$, $b=8$ and $\tG = Q_8 \circ D_8$. There are 6 families of dimension 6 which yield Shimura subvarieties either of $\A^\delta_{4}$ (with $\tg =6$, $g=2$, $b=6$, $\tG = (\Z/2\Z)^2$), or of $\A_{5}$ (with $\tg =11$, $g=6$, $b=0$, $\tG = (\Z/2\Z)^3$). These are all listed in the table in the Appendix and 92 of them satisfy condition \eqref{condB}. The family described in Example 4 is not in the table, since it is shown to yield a 2-dimensional Shimura subvariety of $\A_{12}$ by a different method and it has $\tg = 25$, $g =13$, $b=0$. 
	\end{THEOR} 
		
Moreover for $r=8,9$ the families satisfying condition  \eqref{condB} for $\tg \leq 20$ and $\tG$ abelian are exactly the ones listed in the Table. For $r=7$ we have examples of families  satisfying condition  \eqref{condB} with $b=0,2,4,6,8$; for $r=6$ with $b=0,2,4,6,8, 16$; for $r=5$ with $b=0,2,4,6,8, 12$.

The computations done indicate that for high values of $\tg$ there might not exist such examples of Shimura subvarieties. This is coherent with the computations done in \cite{cfgp}, \cite{fg} and also with the estimates on the dimension of a germ of a totally geodesic submanifold of $\A^\delta_{p}$ contained in the Prym loci found in \cite{cf1}, \cite{cf2}, \cite{cf3}. 

The structure of the paper is as follows: In section 2 we recall the basic definitions and properties of Prym varieties and Prym maps and of special (or Shimura) subvarieties of PEL type. 
In section 3 we describe the construction of the families of Galois covers yielding Shimura subvarieties of  $\A^\delta_{p}$ contained in the Prym loci, and we explain conditions  \eqref{condA}, \eqref{condB} and \eqref{condB1}. In section 4 we concentrate on the case of families abelian covers, we explain their construction following \cite{MZ}, \cite{moh}, \cite{W} and we describe a basis of the space of holomorphic one forms. This will be crucial to make explicit computations on the examples. 
In section 5 we describe some examples of families of covers satisfying conditions \eqref{condA}. The first one has abelian Galois group $\tG$ and satisfies  \eqref{condB1} (hence \eqref{condB}).  The second and the third one have abelian Galois group, they do not satisfy  \eqref{condB1}, and we explicitly show that \eqref{condB} holds using the techniques explained in Section 4. 
	Example 4 has non abelian Galois group and conditon  \eqref{condB1} is not satisfied. In this case we prove that it yields a  Shimura subvariety of  $\A_{12}$ of dimension 2 using \cite [Thm. 3.8]{fgs} and a geometric description of the family. 
	 Finally, in Example 5 we show that condition \eqref{condA} is satisfied, but condition \eqref{condB} is not and we show that it yields a variety which is not totally geodesic, hence it is not Shimura. Notice that in the case of the Torelli map conditions \eqref{condA} and \eqref{condB} are equivalent (see condition $(*)$ in \cite{fgp}), but this does not hold for Prym maps, as it is shown by this example. 
The Appendix contains the table where all the examples found are listed, together with the link to the  \MAGMA\ script.


	\section{Preliminaries on Prym varieties and on special subvarieties of $\ag$}
	\label{Shimura-section}

	\subsection{Prym varieties}
	Denote by $ { \rgb}$ the moduli space of isomorphism classes of triples $[(C, \eta, B)]$ where $C$ is a smooth complex projective curve of genus $g$, $B$ is a reduced effective divisor of degree $b$ on $C$ and $\eta$ is a line bundle on $C$ such that $\eta^2={\mathcal O}_{C}(B)$. This determines a double cover of $C$, $f:\tC\to C$ branched on $B$, with $\tC=\Spec({\mathcal O}_{C}\oplus \eta^{-1})$.
	
	The Prym variety associated to $[(C,\eta, B)]$ is the connected component containing the origin of the kernel of the norm map $\Nm_{f}: J\tC \to JC$.
	If $b>0$, then $\ker \Nm_{f}$ is connected.  It is a  polarised abelian variety of dimension $g-1+\frac{b}{2}$, denoted by $P(C,\eta,B)$ or equivalently by $P(\tC, C)$. 
	
	Denote by $\Xi$ the restriction on $P(\tC,C)$ of the principal polarisation on $J\tC$.
	For $b>0$ the polarisation $\Xi$ is of type $\delta=(1,\dots, 1, \underbrace{2,\dots,2}_{g\text{ times}})$.
	If $b=0$ or $2$ then $\Xi$ is twice a principal polarisation and  $P(\tC, C)$ is endowed with this principal polarisation. Denote by $\A^\delta_{g-1+\frac{b}{2}}$  the moduli space of abelian varieties of dimension $g-1+\frac{b}{2} $ with a polarization of type $\delta$.

	The Prym map $ \pgb$ is defined as follows: 
	\[ \pgb: \rgb \lra \A^\delta_{g-1+\frac{b}{2}}, \quad [(C,\eta,B)] \longmapsto  [(P(C,\eta,B), \Xi)].\]
	
	The Prym locus is the closure in $\A^\delta_{g-1+\frac{b}{2}}$ of the image of the map $\pgb$. 
	
	The dual of the differential of the Prym map $\pgb$ at a generic point $[(C, \eta, B)]$  is given by the multiplication map
	\begin{equation}
	\label{dp}
	(d\pgb)^* : S^2H^0(C, {\omega}_C \otimes \eta) \to H^0(C, \omega^2_C\otimes\mathcal O_C(B))
	\end{equation}
	
	which is known to be surjective as soon as $\dim \rgb \leq \dim {\A}^{\delta}_{g-1+\frac{b}{2}},$ (\cite{lange-ortega}).  Hence $\pgb$ is generically finite, if and only if $\dim \rgb \leq \dim {\A}^{\delta}_{g-1+\frac{b}{2}},$
	and this holds if either $b\geq 6$ and $g\geq 1$, or $b=4$ and $g\geq 3$, $b=2$ and $g\geq 5$,  $b=0$ and $g \geq 6$.
	
	If $b =0$  the Prym map is generically injective for $g \geq 7$  (\cite{friedman-smith}, \cite{kanev-global-Torelli}). 
	If $b>0$, in \cite{pietro-vale},  \cite{mn}, \cite{no}, it is  proved the generic injectivity in all the 
	cases except for $b=4$, $g =3$, which was previously studied in \cite{nagarama}, \cite{bcv} and for which 
	the degree of the Prym map is $3$. 
	Recently, a global Prym-Torelli theorem was proved for all $g$ and $b\ge 6$ (\cite{ikeda} for $g=1$ and \cite{no1} for all $g$).

	\subsection{Special varieties of PEL type}
	\label{VHS}
	Consider a rank $2g$ lattice $\Lambda \cong \Z^{2g}$ and an alternating form $Q : \Lambda \times \Lambda \to \Z $ of type $\delta = (1\dots,1,2,\dots,2)$. 
	There exists a basis of $\Lambda$ such that the corresponding matrix is
	$\begin{pmatrix}
	0 & \Delta_g\\
	-\Delta_g & 0
	\end{pmatrix},$
	where $\Delta_g$ is the diagonal matrix whose entries are $\delta=(1,\dots,1,2,\dots,2)$. 
	Set $U := \Lambda \otimes \R$.  
	The Siegel space $\sieg(U,Q)$ is defined as follows
	\[
	\sieg(U,Q) := \{J \in \GL (U) : J^2 = - I, J^* Q = Q, Q(x,Jx) >0,\ \forall x \neq 0 \}.
	\]
	The symplectic group $\Sp(\Lambda,Q)$ of the form $Q$ acts on the Siegel space $\sieg(U,Q)$ by conjugation and this action is properly discontinuous. 
	The moduli space of abelian varieties of dimension $g$ with a polarisation of type $\delta$ is the quotient 
	$\agd = \Sp(\Lambda,Q) \backslash \sieg(U,Q) $.  
	It is a complex analytic orbifold.	We will consider $\ag^{\delta}$ with the orbifold structure.  
	To an element $J \in \sieg(U,Q)$ we associate the real torus $U/\Lambda \cong  \R^{2g} / \Z^{2g}$ endowed with the complex structure $J$ and the polarization $Q$. This is a polarised abelian variety $A_J$. 
	On $\sieg(U,Q)$ there is a natural variation of rational Hodge structure, whose  local system  is $\sieg(U,Q) \times \Q^{2g}$, and that 
	corresponds to the Hodge decomposition of $\C^{2g}$ in $\pm\i$ eigenspaces for $J$.  
	This gives a variation of Hodge structure on $\agd$ in the orbifold sense.

	A special or Shimura subvariety $\zg \subseteq\agd$ is by definition a Hodge locus of this variation of Hodge structure on $\agd$ (see \cite[ \S 2.3]{moonen-oort} for the definition of Hodge loci). 
	Special subvarieties  are totally geodesic  and they contain a dense set of CM points  \cite[\S 3.4(b)]{moonen-oort}. 
	Conversely an algebraic totally geodesic subvariety that contains a CM point is a special subvariety (\cite{mumford-Shimura}, 
	\cite[Thm. 4.3]{moonen-linearity-1}). 
	
	Let us now recall the definition of  special subvarieties of PEL type (\cite[\S 3.9]{moonen-oort}). 
	Given $J\in \sieg(U,Q)$, set  $\End_\Q (A_{J}) := \{f\in \End( \Q^{2g}): Jf=fJ\}.$
	Fix a point $J_0 \in \sieg(U,Q)$ and consider $D:= \End_\Q (A_{J_0})$.  
	The \emph{PEL type} special subvariety $\zg (D)$ is defined as the image in $\agd$ of the connected component of the set
	$\{J \in \sieg(U,Q): D \subseteq\End_\Q(A_J)\}$ that contains $J_0$.  
	By definition $\zg(D)$ is irreducible.

	If $G\subseteq\Sp(\Lambda, Q)$ is a finite subgroup, denote by $\sieg(U,Q)^G$ the subset of $\sieg(U,Q)$ of fixed points by the action of $G$. 
	Define
	\begin{equation}
	D_G:=\{ f\in \End_\Q (\Lambda \otimes \Q) : Jf=fJ, \ \forall J \in \sieg(U,Q)^G\}.
	\end{equation}
	We have the following result proven in \cite[\S 3]{fgp}. 
	\begin{THEOR}
		\label{bert}
		The subset $\sieg(U,Q)^G$ is a connected complex submanifold of $\sieg(U,Q)$.
		The image of $\sieg(U,Q)^G$ in $\agd$ coincides with the PEL subvariety
		$\zg (D_G)$.  If $J \in \sieg(U,Q)^G $, then
		$ \dim \zg(D_G) = \dim (S^2 \R^{2g})^G$ where $\R^{2g}$ is endowed
		with the complex structure $J$.
	\end{THEOR}

	\section{Special subvarieties in the Prym loci}
	
	Let $\Sigma_g$ be a compact connected oriented surface of genus $g$. 
	We denote by $\T_g := \T(\Sigma_g)$  the Teichm\"uller space of $\Sigma_g$.   
	Denote by $\T_{0,r}$ the Teichm\"uller space in genus $0$  with $r\geq 4$ marked points
(see e.g.  \cite[Chap.\ 15]{acg2}).  Fix $p_0,\dots,p_r\in S^2$ distinct points, denote by  $P:=(p_1,\dots,p_r)$. A point of $\T_{0,r}$ is an equivalence class of triples $(\P^1, x, [h])$ where $x = (x_1, \dots, x_r) $ is an $r$-tuple of distinct points in $\P^1$ and $[h]$ is an isotopy 
	class of orientation preserving homeomorphisms $h : (\P^1, x) \to (S^2 , P)$. Two such triples $(\P^1, x , [h])$, $(\P^1, x', [h'])$ are equivalent if there is a biholomorphism
	$\phi: \P^1 \to \P^1$ such that $\phi(x_i) = x'_i$ for any $i$ and $[h] = [h'\circ \phi ]$. 

	Choosing the  point $p_0 \in S^2 - P$ as base point we have  an isomorphism 
	\begin{equation}
	\label{iso}
	\Gamma_r := \sx \gamma_1, \dots, \gamma_r | \gamma_1 \cdots \gamma_r =1\xs \cong \pi_1(S^2 - P, p_0).
	\end{equation}
	
	Take a point  $t=[(\P^1,x,[h])]\in \T_{0,r}$, and fix  an epimorphism $\theta:\Gamma_r\to G$.
	By Riemann's existence theorem, this gives a Galois cover $C_t\to\P^1=C_t/G$ ramified over $x$ whose monodromy is given by $\theta$,
	which is endowed with an isotopy class of homeomorphisms $[h_t]$ on a surface $\Sigma_g$ covering $S^2$.
	
	So we get a map  $\T_{0,r}\to \T_g:t\mapsto [(C_t,[h_t])]$ and the group $G$ acts on $C_t$ and embeds in the mapping class group $\Mod_g$ of $\Sigma_g$. 
Denote by $G_\theta$ the image of $G$ in $\Mod_g$.
	The image of the map  $\T_{0,r}\to \T_g$ is the subset  $\T_g^{G_\theta}$ of  $\T_g$ given  of the fixed points by the action of $G_\theta$ (\cite{gavino}).
Nielsen realization theorem says that $\T_g^{G_\theta}$ is a complex submanifold  of $\T_g$ of dimension $r-3$.
	
	The image of $ \T^{G_\theta}$ in the moduli space $\mathsf{M}_g$ is a $(r-3)$-dimensional
	algebraic subvariety (see e.g. \cite{gavino,baffo-linceo,clp2} and \cite [Thm. 2.1]{broughton-equi}).

	We denote this image by $\M_g(G, \theta)$. To determine $\M_g(G, \theta)$  we have chosen the isomorphism \eqref{iso}
	$\Gamma_r \cong \pi_1(S^2 -P, p_0)$. To get rid of this choice
 we consider the orbits of the $\braid \times \Aut(G)$--action, called \emph{Hurwitz equivalence classes}, here $\braid$ is the braid group  (see e.g. \cite{baffo-linceo}, \cite{cfgp}). 

	Data in the same orbit give rise to the same subvariety of $\M_g$,
	hence the subvariety $\M_g(G, \theta)$ is well-defined.  
	For more details see \cite{penegini2013surfaces,baffo-linceo,birman-braids}.

	\begin{DEF}
		\label{prymdatum}
		A  Prym datum of type $(r,b)$ is a triple  $\Xi=(\tG,\ttheta,\sigma)$, 
		where $\tG$ is a finite group, $\ttheta:\Gamma_r\to\tG$ is an  epimorphism and  $\sigma\in\tG$ is a central involution and 
		$$b=\sum_{i:\sigma\in\sx{\ttheta(\gamma_i)}\xs}\frac{|\tG|}{\operatorname{ord}(\ttheta(\gamma_i))}.$$
	\end{DEF}
	
	Let us fix a Prym datum $\Xi=(\tG,\ttheta,\sigma)$, and set  $G=\tG/\sx \sigma \xs$.  
	The composition of $\ttheta$ with the projection $\tG \to G$ is an epimorphism $\theta:\Gamma_r\to G$. 
	To  a point $t\in \T_{0,r}$ we can associate  two Galois covers  $\tC_t \to \P^1 \cong \tC_t/\tG$ and $C_t \to \P^1 \cong C_t/G$. 
	Denote by $\tg$ the genus of $\tC_t$ and by $g$ the genus of $C_t$. We have a diagram 
	\begin{equation}
	\label{tc-c}
	\begin{tikzcd}[row sep=tiny]
	\tC_t \arrow{rr}{f_ t} \arrow{rd} & &  \arrow{ld } C_t   = \tC_t /\sx \sigma\xs  \\
	& \P^1 &
	\end{tikzcd}
	\end{equation}
	and by the definition of $b$,  we immediately see that  the double covering $f_t$  has $b$ branch points. 
	The covering $f_t$ is determined by its branch locus $B_t$ and by an element $\eta_t \in \Pic^{\frac{b}{2}}(C_t)$ such that $\eta_t^2 = \mathcal{O}_{C_t}(B_t)$.  Hence we have a map $\T_{0,r}\to\rgb$ which associates to $t$ the class $[(C_t,\eta_t,B_t)]$.
	This map depends on the datum $\Xi=(\tG,\ttheta,\sigma)$. 
	Denote by  $R(\Xi)$ its image. 
	The map $\T_{r,0} \to  R(\Xi)$ has discrete fibres, hence  $\dim R(\Xi)=r-3$.
	
	In \cite[p. 79]{gavino} it is shown that there is an intermediate variety 
	$\widetilde{\M}_{\tg}(\tG, \ttheta)$ such that the projection ${\T}_{\tg}^{\tG_{\tilde{\theta}}} \to  \M_{\tg}(\tG, \tilde{\theta})$  factors through ${\T}_{\tg}^{\tG_{\tilde{\theta}}} \to \widetilde{\M}_{\tg}(\tG, \tilde{\theta}) \to  {\M_{\tg}}(\tG, \tilde{\theta})$. The variety $\widetilde{\M}_{\tg}(\tG, \ttheta)$ is the normalisation of $\M_{\tg}(\tG, \ttheta)$ and it parametrises equivalence classes of curves with an action of $\tG$ of topological type determined by the datum $\Xi$.  Moreover there is a finite cover $ \X_{\tg}(\tG, \ttheta) \to \widetilde{\M}_{\tg}(\tG, \ttheta)$ and a universal family ${\mathcal C}(\tG, \ttheta) \to  \X_{\tg}(\tG, \ttheta)$  \cite[section 4]{gavino}. We also have a factorisation ${\T}_{\tg}^{\tG_{\tilde{\theta}}} \to  \X_{\tg}(\tG, \ttheta) \to \widetilde{\M}_{\tg}(\tG, \tilde{\theta}) \to  {\M_{\tg}}(\tG, \tilde{\theta})$. 

	So we have a map $ \X_{\tg}(\tG, \ttheta) \to R(\Xi)$, 
	which associates to the class $[\tC]$ of a curve with a fixed $\tG$-action the class of the cover $\tC \to C= \tC/ \sx \sigma \xs$, 
	where $\sigma \in \tG$ is the central involution fixed by the datum.  
	Clearly the map $ {\T}_{\tg}^{\tG_{\ttheta}}  \cong \T_{0,r}  \to R(\Xi)$  is the composition  $ {\T}_{\tg}^{\tG_{\tilde{\theta}}} \to \X_{\tg}(\tG, \ttheta) \to R(\Xi)$. 
	Moreover the Prym map $\pgb$ lifts to a map $\X_{\tg}(\tG, \ttheta) \to \A_{\tg-g}^{\delta}$, 
	which sends the class of a curve $[\tC]$ with an action of $\tG$ to $P(\tC, C)$.
	We still denote  this map by $\pgb: \X_{\tg}(\tG, \ttheta) \to \A_{\tg-g}^{\delta}$. 
	We have the following diagram
	
	\[
	\begin{tikzcd}[row sep=small]
	{\T}_g^{G_\theta}\ar[dd] & \T_{0,r}\ar[l,"\cong" swap]\ar[rr,"\cong"] \ar[dd]&[-5ex] &[-5ex]{\T}_{\tg}^{\tG_{\ttheta}}\ar[ld]\ar[dd]\\
	&&{\X}_{\tg}(\tG, \ttheta)\ar[ld]\ar[rd]\\
	\M_g(G , \theta) &R(\Xi)\lar\ar[rr]&&{\M}_{\tg}(\tG, \ttheta) \\ 
	\end{tikzcd}
	\]
	
	Given a Prym datum $\Xi=(\tG,\ttheta,\sigma)$, let us fix an element $\tC$ of the
	family ${\T}_{\tg}^{\tG_{\ttheta}}$ with the double covering
	$f: \tC \to C$. 
	Set
	\[
	V: = H^0(\tC, \omega_{\tC}),
	\]
	
	We have an action of $\sx \sigma \xs$ on $V$ and an eigenspace decomposition for this action: 
	\[ V: = H^0(\tC, \omega_{\tC})= V_+ \oplus V_-,\]
	where $V_+\cong H^0(C, \omega_C)$ and $V_-\cong H^0(C,\omega_C\otimes \eta)$ as $G$-modules.  
	Similarly  set $ W: = H^0( \tC , \omega^2_\tC )= W_+\oplus W_-$,
	$W_+ \cong H^0(C,\omega^2_C \otimes \eta^2) = H^0(C,\omega^2_C\otimes\mathcal O_C(B))$ and
	$W_- \cong H^0(C, \omega^2_C\otimes \eta)$.  
	
	The multiplication map $ m : S^2 V \to W $ is the dual of the differential of the Torelli map
	$\widetilde{j} : \M_\tg \to \A_\tg$ at the point $[\tC] \in \M_\tg$.  
	This map is $\tG$-equivariant.  We have the following isomorphisms
	\[
	(S^2V)^\tG = (S^2V_+)^{\tG} \oplus (S^2 V_-)^{\tG},  \quad
	W^\tG = W_+ ^{\tG}.
	\]
	
	Hence the multiplication map $m$ maps $(S^2V)^\tG $ to $W_+^\tG$.  
	
	Notice that the codifferential of the Prym map  at the point $[(C, \eta, B)]$  is the multiplication map in  \eqref{dp}, 
	which coincides with the restriction of the multiplication map $m$ to $S^2 V_-$. 
	Here we are  interested in the restriction of $m$ to $ (S^2 V_-)^\tG $ that for simplicity we still denote by $m$: 
	\begin{equation}
	\label{nostramram}
	m : (S^2 V_-)^{\tG} \lra W_+^{\tG}.
	\end{equation}
	By what we have said, this is the multiplication map
	\[
	(S^2 H^0 (C, \omega_C\otimes \eta))^G \lra H^0(C, \omega^2_C \otimes \eta^2)^G
	\cong H^0(C,\omega^2_C)^G \cong H^0(\tC,\omega^2_\tC)^{\tG}.
	\]
	
	Now we recall Theorem~3.2 in \cite{fg}, which is a generalisation of Theorems 3.2 and 4.2 in \cite{cfgp}.
	\begin{THEOR}
		\label{teo1ram}
		Let $\Xi=(\tG, \ttheta, \sigma)$ be a Prym datum. If for some $t\in \T_{0,r}$
		the map $m$ in \eqref{nostramram} is an isomorphism, then the
		closure of $\pgb ( \X_{\tg}(\tG, \ttheta))$ in $\A^\delta_{g-1+\frac{b}{2}}$ is a special
		subvariety.
	\end{THEOR}

	We would like to use Theorem \ref{teo1ram} to construct Shimura curves contained in the closure of  $\pgb(\rgb)$ in ${\A}^{\delta}_{g-1+\frac{b}{2}}$ and intersecting $\pgb(\rgb)$. 
	
	The hypothesis in Theorem \ref{teo1ram} is called condition
	\begin{equation}
	\label{condB}
	\tag{B} m : (S^2V_-)^\tG \lra W_+^G \text { is an isomorphism} .
	\end{equation}
	
	This implies  condition 
	\begin{equation}
	\label{condA}
	\tag{A} \dim (S^2V_-)^\tG = r-3.
	\end{equation}
	
	A sufficient condition ensuring
	\eqref{condB} is the following
	
	\begin{equation}
	\label{condB1}
	\tag{B1} (S^2V_-)^\tG \cong Y_1 \otimes Y_2
	\end{equation}
	where $\dim Y_1 =1$, $\dim Y_2 = r-3$. 
	
	In fact, we have the following 
	
	\begin{LEM}
		Assume that \eqref{condB1} holds, then $m : (S^2V_-)^\tG \lra W_+^G$ is injective. 
		Hence since also  \eqref{condA} holds, $m$ is an isomorphism, so condition \eqref{condB} holds.
	\end{LEM}
	\begin{proof}
		Assume \eqref{condB1} holds and take bases for $Y_1$ and $Y_2$: $Y_1 = \sx v \xs$,  $Y_2 = \sx w_1,\dots,w_{r-3} \xs$. 
		Then a basis for $ (S^2V_-)^\tG \cong Y_1 \otimes Y_2$ is given by $\{v \otimes w_i\}$, $i=1,\dots,r-3$. 
		Hence $\forall x \in (S^2V_-)^\tG $, we have $x = \sum_{i=1}^{r-3} a_i (v \otimes w_i )= v \otimes (\sum_{i=1}^{r-3} a_i w_i ) $ is a decomposable tensor. 
		Therefore $m(x) = v \cdot  (\sum_{i=1}^{r-3} a_i w_i ) = 0$  if and only if $0 = \sum_{i=1}^{r-3} a_i w_i$, so $a_i =0$, $\forall i=1,\dots,r-3$, and $x =0$. 
		So $m$ is injective. Hence since also \eqref{condA} holds, $m$ is an isomorphism. 
	\end{proof}
	
	\begin{REM}
		Notice that in the case of one dimensional families  (i.e.\ $r=4$), if the group $\tG$ is abelian and condition \eqref{condA} holds,  then condition \eqref{condB1} (hence also \eqref{condB}) is automatically satisfied ( \cite [Rmk.~3.3]{fg},   \cite [Rmk.~3.4]{cfgp}). This is no longer true for higher dimensional families, as it is shown in Section~5,  Example~\ref{A-not-B}. 
	\end{REM}
	
	Finally in \cite{no1} it is proven that $\pgb$ is an embedding for all $b \geq 6$ and for all $g>0$, hence if $b \geq 6$, and $g>0$ condition  \eqref{condA} implies condition \eqref{condB}.

	\section{Abelian covers of $\P^1$ and their Prym map}
	\label{Abelian covers of line}
	
	In this section, we follow closely \cite{cfgp} and also \cite{MZ} whose notations come mostly from \cite{W}. 
	More details about Prym varieties can be found in \cite{BL}. 
	
	An abelian Galois cover of $\P^1$ is determined by a collection of equations in the following way: 
	
	Consider an $m\times r$ matrix $A=(r_{ij})$ whose entries $r_{ij}$ are in $\Z/N\Z$ for some $N\geq 2$. 
	Let $\overline{\C(x)}$ be the algebraic closure of $\C(x)$. For each $i=1,\dots,m,$ choose a function $w_{i}\in\overline{\C(x)}$ with
	\begin{equation}\label{equation abelian}
	w_{i}^{N}=\prod_{j=1}^{r}(x-t_{j})^{\widetilde{r}_{ij}}\quad\text{for }i=1,\dots, m,
	\end{equation}
	in $\C(x)[w_{1},\dots,w_{m}]$. 
	Here $\widetilde{r}_{ij}$ is the lift of $r_{ij}$ to $\Z \cap [0,N)$ and $t_j\in \C$ for $j=1,\dots, r$. We denote $\widetilde{A}=(\widetilde{r}_{ij})$.
	Notice that \ref{equation abelian} gives in general only a singular affine curve and we take a smooth projective model associated to this affine curve. 
	We impose the condition that the sum of the columns of $A$ is zero (when considered as a vector in $(\Z/N\Z)^m$). 
	This implies that the cover given by \ref{equation abelian} is \emph{not} ramified over the infinity. 
	We call the matrix $A$, the matrix of the covering. We also remark that all operations with rows and columns will be carried out over the ring $\Z/N\Z$, i.e., they will be considered modulo $N$. 
	The local monodromy around the branch point $t_{j}$ is given by the column vector $(r_{1j},\dots, r_{mj})^{t}$ 
	and so the order of ramification over $t_{j}$ is $\frac{N}{\gcd(N,\widetilde{r}_{1j},\dots,\widetilde{r}_{mj})}$. 
	Using this and the Riemann-Hurwitz formula, the genus $g$ of the cover can be computed by:
	\begin{equation}
	g=1+d\left(\frac{r-2}{2}-\frac{1}{2N}\sum_{j=1}^{r}\gcd(N,\widetilde{r}_{1j},\dots,\widetilde{r}_{mj})\right),
	\end{equation}
	where $d$ is the degree of the covering which is equal, as pointed out above, to the column span (equivalently row span)  of the matrix $A$. 
	In this way, the Galois group $G$ of the covering will be a subgroup of $(\Z/N\Z)^{m}$. 
	Note also that this group is isomorphic to the column span of the above matrix.
	We remark that two families of abelian covers with matrices $A$ and $A^{\prime}$ over the same $\Z/N\Z$ such that $A$ and $A^{\prime}$ have equal row spans are isomorphic. \par 
	By the construction of abelian covers of $\P^1$, one can compute the invariants of such a cover quite easily. 
	Let $G$ be a finite abelian group. We denote by $G^*$ the group of the characters of $G$, i.e., $G^*=\Hom(G,\C^{*})$. 
	Consider a Galois covering $\pi: X\rightarrow \P^{1}$ with Galois group $G$. 
	The group $G$ acts on the sheaves $\pi_{*}(\mathcal{O}_X)$ and $\pi_{*}(\C)$ via its characters and we get corresponding eigenspace decompositions 
	$\pi_{*}(\mathcal{O}_X)=\bigoplus_{\chi \in G^*} \pi_{*}(\mathcal{O}_X)_{\chi}$ and $\pi_{*}(\C)=\bigoplus_{\chi \in G^*}\pi_{*}(\C)_{\chi}$. 
	Let $L^{-1}_{\chi}=\pi_{*}(\mathcal{O}_{X})_{\chi}$ and $\C_{\chi}= \pi_{*}(\C)_{\chi}$ denote the eigensheaves corresponding to the character $\chi$. 
	$L_{\chi}$ is a line bundle and outside of the branch locus of $\pi$, $\C_{\chi}$ is a local system of rank 1.  
	\begin{remark} \label{abeliangroupcharacter}
		Let $G$ be a finite abelian group, then the character group $G^*=\Hom(G,\C^{*})$ is isomorphic to $G$. 
		To see this, first assume that $G=\Z/N\Z$ is a cyclic group. 
		Fix an isomorphism between $\Z/N\Z$ and the group of $N$-th roots of unity in $\C^{*}$ via $1\mapsto \exp(2\pi \i/N)$. 
		Now the group $G^*$ is isomorphic to this latter group via $\chi\mapsto \chi(1)$. 
		In the general case, $G$ is a product of finite cyclic groups, so this isomorphism extends to an isomorphism $\varphi_G: G \xrightarrow{\sim} G^*$. 
		In the sequel, we use this isomorphism frequently to identify elements of $G$ with its characters.
	\end{remark}
	
	For our applications, with notations as in the previous pages, 
	we fix an isomorphism of $G$ with a product of $\Z/N\Z$'s and an embedding of $G$ into $(\Z/N\Z)^{m}$.
	
	Let $l_{j}$ be the $j$-th column of the matrix $A$. As mentioned earlier, the group $G$ can be realized as the column span of the matrix $A$. 
	Therefore we may assume that $l_{j}\in G$. 
	For a character $\chi$, $\chi(l_{j})\in \C^{*}$ and since $G$ is finite $\chi(l_{j})$ will be a root of unity.
	Let $\chi(l_{j})=\exp(2\alpha_{j}\pi\i/N)$, where $\alpha_{j}$ is the unique integer in $[0,N)$ with this property. 
	Equivalently, the $\alpha_{j}$ can be obtained in the following way: 
	let $n\in G\subseteq (\Z/N\Z)^{m}$ be the element corresponding to $\chi$ under the above isomorphism. 
	We regard $n$ as an $1\times m$ matrix. Then the matrix product $n\cdotp A$ is meaningful and $n\cdotp A=(\alpha_{1},\dots,\alpha_{r})$. 
	Here all of the operations are carried out in $\Z/N\Z$ but the $\alpha_{j}$ are regarded as	integers in $[0,N)$. Furthermore, we denote by $\widetilde{n}$ the lift of $n$ to $(\Z\cap [0,N))^{m}$ and set $\widetilde{n}\cdotp \widetilde{A}=(\widetilde{\alpha}_{1},\dots,\widetilde{\alpha}_{r})$. In other words $\tilde\alpha_j = \sum_{i=1}^m n_i \tilde{r}_{ij} \in \Z$ (but $\tilde\alpha_j$ is not necessarily in $\Z\cap [0,N)$).
	Using the above facts, we occasionally consider a character of $G$ as an element of this group without referring to isomorphism $\varphi_{G}$.
	
	Let us denote by $\omega_X$ the canonical sheaf of $X$. 
	Similar to the case of $\pi_{*}(\mathcal{O}_X)$, the sheaf $\pi_{*}(\omega_X)_{\chi}$ decomposes according to the action of $G$. 
	For the line bundles $L_{\chi}$ corresponding to the character $\chi$ associated to the element $a\in G$ and $\pi_{*}(\omega_X)_{\chi}$ we have:
	\begin{LEM} \label{eigenbundleformula}
		With notations as above $L_{\chi}=\mathcal{O}_{\P^{1}}(\sum_{j=1}^{r}\langle\frac{\alpha_{j}}{N}\rangle)$, 
		where $\langle x\rangle$ denotes the fractional part of the real number $x$ and
		\[\pi_{*}(\omega_X)_{\chi}= \omega_{\P^1} \otimes L_{\chi^{-1}}=\mathcal{O}_{\P^1}\left(-2+\sum_{j=1}^{r}\left\langle -\frac{\alpha_{j}}{N}\right\rangle\right).\]
	\end{LEM}
	\begin{proof}
		Note that since the sum of the columns of the matrix $A$ is zero, the above sum is an integer. 
		One can easily see that each section of the line bundle $\mathcal{O}_{\P^1}(\sum_{j=1}^{r}\langle \frac{\alpha_{j}}{N}\rangle)$ 
		is a function on which the Galois group acts as $\chi$ and conversely any such section must be a function of the above form. 
		The rest of the lemma is \cite{P2}, Proposition 1.2.
	\end{proof}
	
Consider now an abelian group $\tG\subseteq (\Z/N\Z)^{m}$ containing a central involution $\sigma$ as in section 3 and a $\tG$-abelian cover given by the equations \eqref{equation abelian}. 
	Let $n\in\tG$ be the element $(n_1,\dots, n_m)\in (\Z/N\Z)^{m}$ under the inclusion $\tG\subseteq (\Z/N\Z)^{m}$ with $n_i\in\Z\cap[0,N)$. 
	By Lemma~\ref{eigenbundleformula}, $\dim H^0(\tC,\omega_{\tC})_{n}=-1+\sum_{j=1}^{r}\langle-\frac{\alpha_{j}}{N}\rangle$. 
	A basis for the $\C$-vector space $H^0(\tC,\omega_{\tC})$ is given by the forms 
	\begin{equation}
	\omega_{n,\nu}=x^{\nu} w_{1}^{n_1}\cdots w_{m}^{n_m}\prod_{j=1}^{r} (x-t_j)^{\lfloor -\frac{\tilde\alpha_j}{N}\rfloor}dx.
	\end{equation}
	Here $\tilde\alpha_j$ is as introduced above and $0\leq\nu\leq d_{n}-1=-2+\sum_{j=1}^{r}\langle-\frac{\alpha_{j}}{N}\rangle$. 
	The fact that the above elements constitute a basis can be seen in \cite{MZ}, proof of Lemma 5.1, where the dual version for $H^1(\tC,\mathcal{O}_\tC)$ is proved. Note that in \cite{MZ}, Proposition 2.8, the  formula for $d_n$ has been proven using a different method.
	
	The following lemma is key to our later analysis and shows that we can use the dimension $d_n$ of the eigenspace $H^0(C,\omega_C)_n$ (computed in \cite{MZ}, Proposition 2.8) in our computations with $H^0(C,\omega_C)_{-,n}$. 
	We remark that if $n=(n_1,\dots, n_m)\in G\subset (\Z/N\Z)^{m}$, we consider the $n_i\in [0,N)$ and their sum as integers.

	By the construction of an abelian cover of $\P^1$ at the beginning of this section, the action of $\sigma$ is given by 
	$w_i\mapsto -w_i$ for some subset of $\{1,\dots, m\}$ (and naturally $w_j\mapsto w_j$ for $j$ in the complement of this subset).  
	We may therefore without loss of generality assume that $\sigma(w_i)= -w_i$ for $i\in\{1,\dots,k\}$ for some $k\leq m$ (and $\sigma(w_i)= w_i$ for the $i>k$). The following lemma has been proven in \cite{moh}, Lemma 2.5.
	\begin{LEM} \label{dimeigspace}
		The group $\tG$ acts on $H^0(\tC,\omega_{\tC})_{-}$ and for $n\in\tG$, it holds that $H^0(\tC,\omega_\tC)_{-,n}= H^0(\tC,\omega_\tC)_n$ if $n_1+\cdots+n_k$ is odd and $H^0(\tC,\omega_\tC)_{-,n}=0$ otherwise. Similar equalities hold for $H^1(\tC,\C)_{-,n}$. 
	\end{LEM}

Families of abelian covers of $\P^1$ can be constructed as follows: Denote by $Y_r$ the complement of the big diagonals in $(\mathbb{A}_{\C}^{1})^{r}$, i.e., 
	$Y_r= \{(t_{1},\dots,t_{r})\in(\mathbb{A}_{\C}^{1})^{r}\mid t_{i}\neq t_{j}, \  \forall i\neq j \}$. 
	Over this affine open set we define a family of abelian covers of $\P^1$ by the equation \ref{equation abelian} 
	with branch points $(t_{1},\dots,t_{r})\in Y_r$ and $\widetilde{r}_{ij}$ the lift of $r_{ij}$ to $\Z\cap[0,N)$ as before. 
	Varying the branch points we get a family $f:\tilde{\mathcal C}\to Y_r$ of smooth projective curves whose fibers $\tC_t$ are abelian covers of $\P^1$ introduced above. The subvariety of the moduli space corresponding to this family of curves is  $ \M_{\tg}(\tG, \ttheta) $, with $ \ttheta : \Gamma_r \ra \tG \subseteq (\Z/N\Z)^m$  the epimorphism $\ttheta(\gamma_k) = A_k$, where $A_k$ is the $k$-th column of the matrix $A=(r_{ij})$.

	\section{Examples in the Prym locus}
	
	In this section we describe some examples of families of covers satisfying conditions \eqref{condA}. The first one has abelian Galois group $\tG$ and satisfies  \eqref{condB1} (hence \eqref{condB}).  The second and the third one have abelian Galois group, they do not satisfy  \eqref{condB1}, and we explicitly show that \eqref{condB} holds using the techniques explained in Section 4. 
	Example 4 has non abelian Galois group and conditon   \eqref{condB1} is not satisfied. In this case we prove that it yields a  Shimura subvariety of  $\A_{12}$ of dimension 2 using \cite [Thm. 3.8]{fgs} and a geometric description of the family. 
	Example 5 satisfies condition \eqref{condA} but not condition \eqref{condB} and we show that it yields a variety which is not totally geodesic, hence it is not Shimura. \\
	
	To describe the examples we first give the values of $\tg$, $g$ and $b$. We also indicate a number $\#$ that refers to the numbering of the examples that we found with the  \MAGMA\ script with the given values of $r$, $\tg$, $g$ and $b$.  For the group $\tG$ we use the presentation  and the name of the group given by  \MAGMA. 
	In Example 5 we give the decomposition of $V_+$ and $V_{-}$ as a direct sum of irreducible representations of $\tG$ and  also  for the  irreducible representations we use the notation of \MAGMA. \\
	
	The first example that we describe satisfies condition \eqref{condB1}. 
	
	\begin{example}
		$r=5$, $\tg = 12$, $g=6$, $b=2$. \\
		$\tG = G(10,2) = \Z/10\Z = \langle g_1 \rangle$. \\
		$(\ttheta(\gamma_1),\dots,\ttheta(\gamma_5))=(g_1,g_1,g_1^2,g_1^2,g_1^4),  \quad\sigma=g_1^5$. \\
		This is the family with equation given by 
		$y^{10}=(x-t_1)(x-t_2)(x-t_3)^2(x-t_4)^2(x-t_5)^4$ and 
		$\sigma(y) = -y$. Using the notation of the previous section, the matrix giving the monodromy is 
		$(1,1,2,2,4)$. Hence we have $d_1 = 3$, $d_3 = 2$, $d_7 =1$, $d_9 = 0$, where $d_i = \dim(W_i)$, $W_i = \{ \omega \in H^0(K_{\tC}) \ | \ g_1 (\omega) = \xi_{10}^{-i} \omega\}$, with $\xi_{10}$ a primitive $10^{th}$ root of unity. So we have $V_- = W_1 \oplus W_{3} \oplus W_7$, hence $(S^2V_-)^{\tG} = W_3 \otimes W_7$ and $\dim(W_7) = 1$, therefore condition \eqref{condB1} holds. So the family of Pryms yields a Shimura subvariety $ \overline{{\mathsf{P}_{6,2}}(\X_{12}(\tG, \tilde{\theta}))}$ of ${\mathsf{A}}_6$. 
	\end{example}
	
	The next  two examples do not satisfy condition \eqref{condB1} but we show that they satisfy condition \eqref{condB}. 
	
	\begin{example}
		$r=6$, $\tg = 5$, $g=3$, $b=0$. \\
		$\tG=G(8,5)=\ZZ/2\times \ZZ/2 \times \ZZ/2=\langle g_1 \rangle  \times \langle g_2 \rangle \times \langle g_3 \rangle,$\\
		$(\ttheta(\gamma_1),\dots,\ttheta(\gamma_6))=(g_1,g_1,g_2,g_2,g_3,g_3),  \quad\sigma=g_1g_2g_3$. \\
		
		Hence the matrix giving the monodromy is 
		$$A=  {\small \left( \begin{array}{cccccc}
			1&1&0&0&0&0\\
			0&0&1&1&0&0\\
			0&0&0&0&1&1\\
			\end{array} \right)},
		$$
		and the equations for the curves $\tC$ in the family are: 
		\begin{gather}
		w_1^{2}=(x-t_1)(x-t_2)\\
		w_2^{2}=(x-t_3)(x-t_4)\nonumber\\
		w_3^{2}=(x-t_5)(x-t_6)\nonumber
		\end{gather}
		The action of $\sigma$ is in this case given by $w_i\mapsto -w_i$ for $i=1,2,3$.
		One sees that $H^0(\tC,\omega_{\tC})_-=H^0(\tC,\omega_{\tC})_{-,(1,1,1)}=H^0(\tC,\omega_{\tC})_{(1,1,1)}$.
		By Lemma \ref{dimeigspace}, we know that $d_{(1,1,1)} = \dim H^0(\tC,\omega_{\tC})_{-,(1,1,1)}=2$. 
		Hence $\dim(S^2V_-)^G=3$.
		In order to show that the family gives rise to a 3-dimensional special subvariety in $A_5$, we show that condition \eqref{condB} holds.
		We have that
		\[ V_- = H^0(\tC,\omega_{\tC})_{(1,1,1)}=\langle\alpha_1=w_1w_2w_3\prod_{i=1}^{6}(x-t_i)^{-1}dx = \frac{dx}{w_1w_2w_3}, \alpha_2=x\alpha_1\rangle,\]
		so that $(S^2V_-)^G=\langle\alpha_1\odot\alpha_1,\alpha_1\odot\alpha_2,\alpha_2\odot\alpha_2\rangle$. 
		We have
		\begin{gather*}
		m(\alpha_1\odot\alpha_1)=\frac{(dx)^2}{\prod_{i=1}^{6}(x-t_i)}, \ 
		m(\alpha_1\odot\alpha_2)=\frac{x(dx)^2}{\prod_{i=1}^{6}(x-t_i)}, \
		m(\alpha_2\odot\alpha_2)=\frac{x^2(dx)^2}{\prod_{i=1}^{6}(x-t_i)}.
		\end{gather*}
		So $v=a_1(\alpha_1\odot\alpha_1)+a_2(\alpha_1\odot\alpha_2)+a_3(\alpha_2\odot\alpha_2)\in\ker(m)$
		if and only if
		\[a_1\frac{(dx)^2}{\prod_{i=1}^{6}(x-t_i)}+a_2\frac{x(dx)^2}{\prod_{i=1}^{6}(x-t_i)}+
		a_3\frac{x^2(dx)^2}{\prod_{i=1}^{6}(x-t_i)}=0.\]
		It is straightforward to see that this holds if and only if
		$a_1=a_2=a_3=0$. This shows that $m$ is injective and by condition
		\eqref{condA}, it is an isomorphism, so condition \eqref{condB} is satisfied. So we have shown that  $\overline{{\mathsf{P}_{3,0}}(\X_{5}(\tG, \tilde{\theta}))} \subset {\mathsf{A}}_2$ is a Shimura subvariety of dimension 3, hence it  coincides with ${\mathsf{A}}_2$. 
		
		Notice that $P(\tC, C) \sim JC'$, where $C' = \tC/\langle g_1g_2, g_1g_3 \rangle$ is a genus 2 curve whose equation as a double cover of $\PP^1$ via the map $C' \ra C'/H$, $H  = \tG/\langle g_1g_2, g_1g_3 \rangle$, is $(w_1w_2w_3)^2=\prod_{i=1}^{6}(x-t_i) $, so that $H^{1,0}(C') = \langle \alpha_1, \alpha_2, \alpha_3\rangle = V_-$. 
	\end{example}

	\begin{example}
		$r=9$, $\tg=11$, $g=6$, $b=0$.\\
		$\tG=G(8,5)=\ZZ/2\times \ZZ/2 \times \ZZ/2=\langle g_1 \rangle  \times \langle g_2 \rangle \times \langle g_3 \rangle,$\\
		$(\ttheta(\gamma_1),\dots,\ttheta(\gamma_9))=(g_1, \  g_1g_2g_3, \ g_3, \ g_2, \ g_2g_3, \ g_1g_2, \ g_2g_3, \ g_2, \ g_1), \quad\sigma=g_1g_3$.\\
		From the MAGMA script we know that $\dim(S^2V_-)^{\tG}= 6$, hence condition  \eqref{condA} is satisfied, but \eqref{condB1} is not. We will show that  condition \eqref{condB} holds. 
		
		Denote by $(t_1,\dots,t_9)$ the critical values of the covering $\psi: \tC \ra \tC/\tG \cong \PP^1$. One can easily show that $P(\tC, C) \sim E \times F \times Y \times JX$, where $E$, $F$, $Y$ have genus 1, while $X$ has genus 2. More precisely, the elliptic curve $E$ is the quotient $E = \tC/\langle g_1, g_2 \rangle$, and  if we set $H = \tG/\langle g_1, g_2 \rangle \cong \ZZ/2$, the double cover $E \ra E/H \cong \PP^1$ ramifies over $\{t_2,t_3, t_5, t_7\}$, hence $E$ has an equation given by $y^2 = (x-t_2)(x-t_3)(x-t_5)(x-t_7)$. The elliptic curve $F$ is the quotient $F = \tC/ \langle g_2, g_3 \rangle$ and the double cover $F \ra F/K \cong \PP^1$ where $K = \tG/ \langle g_2, g_3 \rangle$ has an equation given by $\eta^2 = (x-t_1)(x-t_2)(x-t_6)(x-t_9)$. The elliptic curve $Y$ is the quotient $Y= \tC/ \langle g_1, g_2g_3 \rangle$ and the double cover $Y \ra F/L \cong \PP^1$ where $L = \tG/ \langle g_1, g_2g_3 \rangle$ has an equation given by $w^2 = (x-t_3)(x-t_4)(x-t_6)(x-t_8)$. Finally the genus 2 curve $X$ is the quotient $X= \tC/ \langle g_1g_2, g_3 \rangle$ and the double cover $X \ra F/N \cong \PP^1$ where $N = \tG/ \langle g_1g_2, g_3 \rangle$ has an equation given by $\theta^2 = (x-t_1)(x-t_4)(x-t_5)(x-t_7)(x-t_8)(x-t_9)$. 
		Hence we have $(S^2V_-)^{\tG} \cong S^2 H^{1,0}(E) \oplus S^2 H^{1,0}(F) \oplus S^2 H^{1,0}(Y) \oplus S^2 H^{1,0}(X)$. 
		
		The matrix giving the monodromy is the following 
		$$A=  {\small \left( \begin{array}{ccccccccc}
			1&1&0&0&0&1&0&0&1\\
			0&1&0&1&1&1&1&1&0\\
			0&1&1&0&1&0&1&0&0\\
			\end{array} \right)},
		$$
		so the equations of $\tC$ are 
		\begin{gather*}
		w_1^2 = (x-t_1)(x-t_2)(x-t_6)(x-t_9), \\
		w_2^2 = (x-t_2)(x-t_4)(x-t_5)(x-t_6)(x-t_7)(x-t_8),\\
		w_3^2 = (x-t_2)(x-t_3)(x-t_5)(x-t_7).
		\end{gather*}
		
		Notice that the first equation gives the elliptic curve $F$, and the third one gives $E$. 
		The action of $\sigma = g_1g_3$ is the following $\sigma(w_1) = -w_1$, $\sigma(w_2) = w_2$, $\sigma(w_3) = -w_3$.  One immediately computes $d_{g_1} = d_{(1,0,0)} = 1$ and $\alpha_1:= \omega_{g_1,0} = \frac{dx}{w_1} $ is a generator of $ H^{1,0}(F)$. We have: $d_{g_3} = d_{(0,0,1)} = 1$ and $\alpha_2:= \omega_{g_3,0} = \frac{dx}{w_3} $ is a generator of $ H^{1,0}(E)$. We have $d_{g_1g_2} = d_{(1,1,0)} = 2$ and $\alpha_3:= \omega_{g_1g_2,0} = \frac{(x-t_2)(x-t_6)dx}{w_1w_2} $, $\alpha_4:= \omega_{g_1g_2,1} = \frac{x(x-t_2)(x-t_6)dx}{w_1w_2}  = x \alpha_3$. Moreover $ H^{1,0}(X) = \langle \alpha_3, \alpha_4\rangle$, since if we set $\theta:= \frac{w_1 w_2}{(x-t_2)(x-t_6)}$, we have $\theta^2 = (x-t_1)(x-t_4)(x-t_5)(x-t_7)(x-t_8)(x-t_9)$, hence $\langle \alpha_3 = \frac{dx}{\theta},$ $\alpha_4 = \frac{x dx}{\theta} \rangle =  H^{1,0}(X)$. 
		Finally one computes $d_{g_2g_3} = d_{(0,1,1)} = 1$ and $\alpha_5:= \omega_{g_2g_3,0} = \frac{(x-t_2)(x-t_5)(x-t_7)dx}{w_2w_3} $ is a generator of $ H^{1,0}(Y)$, since $w:= \frac{w_2w_3}{(x-t_2)(x-t_5)(x-t_7)} $ satisfies $w^2 = (x-t_3)(x-t_4)(x-t_6)(x-t_8)$ and $ H^{1,0}(Y) = \langle \frac{dx}{w}= \alpha_5 \rangle$. 
		
		Therefore we have 
		\begin{equation}
		(S^2V_-)^{\tG} = \langle \alpha_1 \odot \alpha_1, \alpha_2 \odot \alpha_2, \alpha_3 \odot \alpha_3, \alpha_3 \odot \alpha_4, \alpha_4 \odot \alpha_4, \alpha_5 \odot \alpha_5 \rangle
		\end{equation}
		We have 
		\begin{gather*}
		m(\alpha_1 \odot \alpha_1) = \alpha_1^2  = \frac{(dx)^2}{w_1^2} = \frac{(dx)^2}{(x-t_1)(x-t_2)(x-t_6)(x-t_9)},\\
		m(\alpha_2 \odot \alpha_2) = \alpha_2^2= \frac{(dx)^2}{w_3^2} = \frac{(dx)^2}{(x-t_2)(x-t_3)(x-t_5)(x-t_7)},\\
		m(\alpha_3 \odot \alpha_3) = \alpha_3^2  =\frac{(x-t_2)^2(x-t_6)^2(dx)^2}{w_1^2w_2^2}= \frac{(dx)^2}{(x-t_1)(x-t_4)(x-t_5)(x-t_7)(x-t_8)(x-t_9)},\\
		m(\alpha_3 \odot \alpha_4) = \alpha_3\alpha_4= x m(\alpha_3 \odot \alpha_3), \\
		m(\alpha_4 \odot \alpha_4) = \alpha_4^2 =x^2 m(\alpha_3 \odot \alpha_3), \\
		m(\alpha_5 \odot \alpha_5) =  \alpha_5^2 = \frac{(x-t_2)^2(x-t_5)^2(x-t_7)^2(dx)^2}{w_2^2w_3^2}= \frac{(dx)^2}{(x-t_3)(x-t_4)(x-t_6)(x-t_8)}.
		\end{gather*}
		So, $v = a_1 (\alpha_1 \odot \alpha_1) + a_2 (\alpha_2 \odot \alpha_2) + a_3  (\alpha_3 \odot \alpha_3) + a_4  (\alpha_3 \odot \alpha_4) + a_5 (\alpha_4 \odot \alpha_4) + a_6 (\alpha_5 \odot \alpha_5) \in \ker(m)$ if and only if $a_1 \alpha_1^2+ a_2 \alpha_2^2 + a_3 \alpha_3^2+ a_4  \alpha_3 \alpha_4 + a_5 \alpha_4^2 + a_6 \alpha_5^2=0$ and this holds if and only if 
		$$a_1(x-t_3)(x-t_4)(x-t_5)(x-t_7)(x-t_8) + a_2(x-t_1)(x-t_4)(x-t_6)(x-t_8)(x-t_9) + $$
		$$a_3(x-t_2)(x-t_3)(x-t_6) + a_4 x   (x-t_2)(x-t_3)(x-t_6) + a_5 x^2 (x-t_2)(x-t_3)(x-t_6) + $$
		$$a_6 (x-t_1)(x-t_2)(x-t_5)(x-t_7)(x-t_9) = 0,$$
		and one can easily show that the only solution is $a_i = 0$, $\forall i =1,...,6$. This proves that $m$ is injective, hence by condition \eqref{condA}, it is an isomorphism, so condition \eqref{condB} is satisfied.
	\end{example}

	Now we describe an example with non abelian group $\tG$ that does not satisfy \eqref{condB1}, and we show that it gives a 2-dimensional Shimura subvariety of ${\mathsf{A}}_{12}$ using \cite [Thm. 3.8]{fgs}. \\
	\begin{example}
		$r=5$, $\tg=25$, $g=13$, $b=0$.\\
		$\tG=G(48,32)=\ZZ/2\times SL(2,3)=\sx{g_1|g_1^2=1}\xs\times \langle g_2,g_3,g_4, g_5 \ | \ g_1^2=g_2^3 =1, \ g_3^2=g_4^2=g_5, \ g_5^2 =1,$
		
		$g_2^{-1}g_3g_2=g_4, \ g_2^{-1}g_4g_2=g_3g_4, \ g_3^{-1}g_4g_3=g_4g_5 \rangle$,\\
		$(\ttheta(\gamma_1)=g_1g_5,\ \ttheta(\gamma_2)=g_1,\ \ttheta(\gamma_3)=g_2g_4,\ \ttheta(\gamma_4)=g_2g_3g_4, \ttheta(\gamma_5)=g_2), \quad\sigma=g_5$.\\
		Assume that the critical values of the map $\psi: \tC \ra \tC/\tG \cong \PP^1$ are $(t_1, t_2, 0, 1,\infty)$. 
		$V_+=V_4\oplus V_5\oplus 2V_6\oplus 3V_{13}$ and $V_-=2V_{7}\oplus 2V_{8} \oplus V_9 \oplus V_{11}$, where $\dim V_i=1$, $\forall i <7$, $ \dim V_7 = \dim V_{8} = \dim V_{9 } = \dim V_{11} =2$, $\dim V_{13} =3$ . 
		$(S^2V_-)^\tG\cong \Lambda^2V_7 \oplus  \Lambda^2V_8$ has dimension 2. Thus condition \eqref{condA} is satisfied, but \eqref{condB1} is not.
		We want to show that the family yields a Shimura subvariety of $\A_{12}$ of dimension 2. 
		Notice that the centre $Z$ of $\tG$ is the subgroup $Z= \langle g_1, g_5 \rangle$, and $V_{4}=\Fix\sx{g_1, g_5}\xs$, hence $E:= \tC/ Z$ is a genus 1 curve and it  gives a  degree 12 Galois cover $\phi: E \ra  E/L = \tC/\tG \cong {\mathbb P}^1$, where $L= \tG/ Z$ with $H^{1,0}(E) = V_4$.  It is immediate to check that the  Galois cover $\phi$ only ramifies over $(0,1, \infty)$, hence the elliptic curve $E$ does not move. 

		We also have $\Fix\sx{g_1}\xs = V_4 \oplus 2V_7 \oplus V_{11}$, $\Fix\sx{g_1g_5} \xs= V_4 \oplus 2V_8 \oplus V_{9}$, hence if we set $C' = \tC/\langle g_1 \rangle$, $D'  = \tC/\langle g_1g_5 \rangle$, we have $H^{1,0} (C') = V_4 \oplus 2V_7 \oplus V_{11}$, $H^{1,0} (D')= V_4 \oplus 2V_8 \oplus V_{9}$. Both $\langle g_1 \rangle$ and $\langle g_1g_5 \rangle$ are normal subgroups and both quotients $H:= \tG/\langle g_1 \rangle$ and $K:= \tG/\langle g_1g_5\rangle$ are isomorphic to $SL(2,3)$.  So we have two $SL(2,3)$-Galois covers $\alpha: C' \ra C'/H= \tC/\tG \cong \PP^1$, $\beta: D' \ra C'/K = \tC/\tG \cong \PP^1$ and a commutative diagram

		\[\begin{tikzcd}
		&\ar[ld]\tC\ar[rd]\ar[dd]\\
		\tC/\sx{g_1}\xs=C'\ar[rd]\ar[rdd, "\alpha"]&&\tC/\sx{g_1g_5}\xs=D'\ar[ld]\ar[ldd, "\beta"]\\
		&\tC/Z = E \ar[d]\\
		&\tC/\tG = \PP^1\\
		\end{tikzcd}\]
		
		The two families of Galois covers $\alpha: C'_{t_1} \ra C'_{t_1}/H \cong \PP^1$  and $\beta: D'_ {t_2} \ra D'_{t_2} /K \cong \PP^1$ yield Shimura curves. In fact they both satisfy condition $(*)$ of \cite{fgp} since $(S^2 H^0(K_{C'}))^H \cong \Lambda^2 V_7$ and  $(S^2 H^0(K_{D'}))^K \cong \Lambda^2 V_8$, which are both one dimensional, and they coincide with family $(40)$ of \cite{fgp}. The critical values of $\alpha$ and $\beta$ are respectively $(t_1, 0, 1, \infty)$, $(t_2, 0, 1, \infty)$ with the same monodromy. We have isogenies $JC'_{t_1} \sim E \times P(C'_{t_1},E)$, $JD'_{t_2} \sim E \times P(D'_{t_2},E)$, where $E$ is fixed. Therefore the one dimensional families of Pryms $P(C'_{t_1},E)$ and $P(D'_{t_2},E)$ yield Shimura curves in the moduli space of polarised abelian varieties $\A_6^{\delta}$. Denote by $W_1 := \overline{{\mathsf{P}_{13,0}}(\X_{25}(\tG, \tilde{\theta}))}$ and consider the two dimensional Shimura subvariety $W_2$ of $\A_6^{\delta} \times \A_6^{\delta} \subset \A_{12}^{\delta'}$ given by the family $P(C'_{t_1},E) \times P(D'_{t_2},E)$. For any $(t_1, t_2) \in  (\PP^1\setminus \{0,1, \infty\}) \times (\PP^1\setminus \{0,1, \infty\})$ with $t_1 \neq t_2$, we have an isogeny  $P(\tC_{(t_1, t_2)}, C_{(t_1, t_2)}) \sim P(C'_{t_1},E) \times P(D'_{t_2},E)$. Hence  there exists an open subset $U$ of $W_2$ such that for any  abelian variety $A \in W_2$ there exists an abelian variety $B \in W_1$ such that $A$ is isogenous to $B$. So, applying  \cite [Thm. 3.8]{fgs} we see that $2= \dim (W_2) \leq \dim(W_1) \leq 2$, therefore $\dim(W_1) =2$ and  by  \cite [Thm. 3.8]{fgs} it is also totally geodesic, since $W_2$ is totally geodesic.  Moreover $W_2$ is  Shimura, hence there exist infinitely many $t_1$ and   $t_2$ such that both $P(C'_{t_1},E)$  and $P(D'_{t_2},E)$ have complex multiplication. So also the family of Pryms $P(\tC_{(t_1, t_2)}, C_{(t_1, t_2)})$ has a CM point, hence $W_1$ is a Shimura subvariety of $\A_{12}$ of dimension 2. 
	\end{example}\\
	
	Finally we give an example satisfying condition \eqref{condA} but not condition \eqref{condB} and we show that it yields a variety which is not totally geodesic, hence it is not Shimura. \\
	
	\begin{example}
		\label{A-not-B}%
		$r =10$, $\tg=7$, $g=3$, $b=4$.\\
		$\tG=G(4,2)=\ZZ/2\times \ZZ/2 =\langle g_1 \rangle  \times \langle g_2 \rangle,$\\
		$(\ttheta(\gamma_1),\dots,\ttheta(\gamma_{10}))=(g_1, \  g_1, \ g_1g_2, \ g_2, \ g_1, \ g_1, \ g_1, \ g_1g_2, \ g_2, g_1), \quad\sigma=g_1g_2$.\\
		From the MAGMA script we know that $\dim(S^2V_-)^{\tG}= 7$, hence condition  \eqref{condA} is satisfied, but \eqref{condB1} is not. We will show that  condition \eqref{condB} does not hold and that the family is not totally geodesic, hence it is not Shimura. 
		
		Denote by $(t_1,\dots,t_{10})$ the critical values of the covering $\psi: \tC \ra \tC/\tG \cong \PP^1$. One can easily show that $P(\tC, C) \sim E \times JC' $, where $C'$ has genus 3, while $E$ has genus 1. More precisely, the genus 3 curve $C'$ is the quotient $C'= \tC/ \langle g_2\rangle$ and the double cover $C' \ra C'/N \cong \PP^1$ where $N = \tG/ \langle g_2 \rangle$ has an equation given by $w_1^2 = (x-t_1)(x-t_2)(x-t_3)(x-t_5)(x-t_6)(x-t_7)(x-t_8)(x-t_9)$. The elliptic curve $E$ is the quotient $E = \tC/\langle g_1 \rangle$, and  if we set $H = \tG/\langle g_1\rangle \cong \ZZ/2$, the double cover $E \ra E/H \cong \PP^1$ ramifies over $\{t_3,t_4, t_8, t_9\}$, hence $E$ has an equation given by $w_2^2 = (x-t_3)(x-t_4)(x-t_8)(x-t_9)$. 
		Therefore we have $(S^2V_-)^{\tG} \cong S^2 H^{1,0}(E) \oplus S^2 H^{1,0}(C')$ and it has dimension 7.  
		
		The matrix giving the monodromy is the following 
		$$A=  {\small \left( \begin{array}{cccccccccc}
			1&1&1&0&1&1&1&1&0&1\\
			0&0&1&1&0&0&0&1&1&0\\
			\end{array} \right)},
		$$
		so the equations of $\tC$ are 
		\begin{gather*}
		w_1^2 = (x-t_1)(x-t_2)(x-t_3)(x-t_5)(x-t_6)(x-t_7)(x-t_8)(x-t_{10}), \\
		w_2^2 = (x-t_3)(x-t_4)(x-t_8)(x-t_9).
		\end{gather*}
		
		Notice that the first equation gives the curve $C'$, while the second one gives $E$. 
		
		The action of $\sigma = g_1g_2$ is the following $\sigma(w_1) = -w_1$, $\sigma(w_2) = -w_2$.  One immediately computes $d_{g_2} = d_{(0,1)} = 1$ and $\alpha_1:= \omega_{g_1,0} = \frac{dx}{w_2} $ is a generator of $ H^{1,0}(E)$. We have: $d_{g_1} = d_{(1,0)} = 3$ and $\{\alpha_2:= \omega_{g_1,0} = \frac{dx}{w_1}, \alpha_3:= \omega_{g_1,1} = \frac{xdx}{w_1}, \alpha_4:= \omega_{g_1,2} = \frac{x^2dx}{w_1}\}$ is a basis of $ H^{1,0}(C')$.
		Therefore we have 
		\begin{equation}
		(S^2V_-)^{\tG} = \langle \alpha_1 \odot \alpha_1, \alpha_2 \odot \alpha_2, \alpha_2 \odot \alpha_3,  \alpha_2 \odot \alpha_4, \alpha_3 \odot \alpha_3, \alpha_3 \odot \alpha_4, \alpha_4 \odot \alpha_4 \rangle
		\end{equation}
		Moreover, 
		\begin{gather*}
		m(\alpha_1 \odot \alpha_1) = \alpha_1^2  = \frac{(dx)^2}{w_2^2} = \frac{(dx)^2}{(x-t_3)(x-t_4)(x-t_8)(x-t_9)}, \\
		m(\alpha_2 \odot \alpha_2) =  \alpha_2^2 = \frac{(dx)^2}{(x-t_1)(x-t_2)(x-t_3)(x-t_5)(x-t_6)(x-t_7)(x-t_8)(x-t_{10})},\\
		m(\alpha_2 \odot \alpha_3) = \alpha_2\alpha_3= x\alpha_2^2, \
		m(\alpha_2 \odot \alpha_4) = \alpha_2 \alpha_4 =x^2\alpha_2^2,\
		m(\alpha_3 \odot \alpha_3) = \alpha_3^2 =x^2\alpha_2^2,\\
		m(\alpha_3 \odot \alpha_4) = \alpha_3 \alpha_4 =x^3\alpha_2^2,\
		m(\alpha_4 \odot \alpha_4) = \alpha_4^2 =x^4\alpha_2^2.
		\end{gather*}
		So, $Q = a_1 (\alpha_1 \odot \alpha_1) + a_2 (\alpha_2 \odot \alpha_2) + a_3  (\alpha_2 \odot \alpha_3) + a_4  (\alpha_2 \odot \alpha_4) + a_5 (\alpha_3 \odot \alpha_3) + a_6 (\alpha_3\odot \alpha_4) + a_7(\alpha_4 \odot \alpha_4) \in \ker(m)$ if and only if $a_1 \alpha_1^2+ a_2 \alpha_2^2 + a_3 \alpha_2\alpha_3+ a_4  \alpha_2\alpha_4  + a_5 \alpha_3^2 + a_6 \alpha_3\alpha_4 + a_7 \alpha_4^2=0$ and this holds if and only if 
		$$a_1(x-t_1)(x-t_2)(x-t_5)(x-t_6)(x-t_7)(x-t_{10}) + (a_2 + a_3 x + (a_4 + a_5)x^2 + a_6 x^3 + a_7 x^4)(x-t_4)(x-t_9)  = 0. $$
		One can verify that this holds if and only if $a_1=a_2=a_3=a_6=a_7 =0, a_5 = -a_4$, hence 
		$$\ker(m) = \langle Q=\alpha_2 \odot \alpha_4 -\alpha_3 \odot \alpha_3 \rangle,$$
		so condition \eqref{condB} is not satisfied. 
		
		Notice that since $m$ has rank 6, the variety $\overline{{\mathsf{P}_{3,4}}(\X_{7}(\tG, \tilde{\theta}))}=:W_1$ has dimension 6. 
		
		Now we show that  $W_1 \subset {\mathsf{A}}_4^{\delta}$ is not totally geodesic. 
		
		Consider the subvariety $W_2$ of ${\mathsf{A}}_4$ given by ${\mathsf{A}}_1 \times \overline{j(HE_3)}$, where $HE_3$ denotes the hyperelliptic locus in ${\mathsf{M}}_3$ and $j: {\mathsf{M}}_3 \ra {\mathsf{A}}_3$ is the Torelli map. We have shown that there exists an open subset $U$ of $W_1$ such that for every abelian variety  $A= P(\tC,C) \in U$, there exists an abelian variety $B \in W_2$ such that $A$ is isogenous to $B=E \times JC'$.  Since $W_1$ and $W_2$ have the same dimension, \cite [Thm. 3.8]{fgs} implies that $W_1$ is totally geodesic in $ {\mathsf{A}}_4^{\delta}$ if and only if $W_2$ is totally geodesic in ${\mathsf{A}}_4$. 
		So, assume by contradiction that $W_1$ is totally geodesic in $ {\mathsf{A}}_4^{\delta}$, then $W_2 \subset {\mathsf{A}}_1 \times {\mathsf{A}}_3 \subset {\mathsf{A}}_4$ would be totally geodesic in ${\mathsf{A}}_4$, hence it  also would be  totally geodesic in ${\mathsf{A}}_1 \times {\mathsf{A}}_3$, since ${\mathsf{A}}_1 \times {\mathsf{A}}_3$ is totally geodesic in ${\mathsf{A}}_4 $. This would imply that $\overline{j(HE_3)}$ is totally geodesic in ${\mathsf{A}}_3$, which is not true, as one can see applying \cite [Corollary 5.14]{fgp}, since the hyperelliptic locus has codimension 1 in ${\mathsf{M}}_3$.    
		
		Another way to prove that $\overline{j(HE_3)}$ is not totally geodesic in ${\mathsf{A}}_3$ is via a direct computation of the second fundamental form of the Torelli map that can be done as in  \cite [Proposition 5.4]{cfg}. In fact, with the notation of  \cite [Proposition 5.4]{cfg} one can show that $\rho(Q)(v \odot v) \neq 0$, where $v$ is a tangent vector to $W_2$ at $JC'$ given by the sum of two Schiffer variations $v=\xi_p + \xi_q$, at two general  distinct points $p, q \in  C'$ that are exchanged by the hyperelliptic involution. 
		
		Hence we have proven that $W_1= \overline{{\mathsf{P}_{3,4}}(\X_{7}(\tG, \tilde{\theta}))}$ is not totally geodesic in  ${\mathsf{A}}_4^{\delta}$, therefore it is not Shimura.


	\end{example}

	\section*{Appendix}
	
	The  \MAGMA\  script used to find the examples is available at:

\centerline{\texttt{http://www-dimat.unipv.it/grosselli/publ/}}

\noindent and it is described in the Appendix in \cite{fg}. The script tests condition \eqref{condA}, next if either \eqref{condB1} is met or $b\ge6$ then condition \eqref{condB} holds.
The method described in Section~\ref{Abelian covers of line} is used to test \eqref{condB} for many specific abelian cases. In particular in the cases with $r=8,9$ the families satisfying condition  \eqref{condB} for $\tg \leq 20$ and $\tG$ abelian are exactly the ones listed in the Table. 

	The following table lists all the calculated cases in which condition \eqref{condB} holds. 
The data are listed sorted by $r$, $\tg$ and $g$ (and thus $b$ and $p = \tg-g$), then a progressive index (under column $\#$) allows to distinguish cases with the same values. 
	For each case it is reported the group $\tG$ and the id of $\tG$ and $G$ as small group in the \MAGMA\ Database. A check mark tells if the conditions are met.
	Different examples with same data are grouped in the same row.


\def\midline{\addlinespace[0mm]\hline\addlinespace[0.379mm]}
\def\cmidline{\addlinespace[0mm]\cline{2-12}\addlinespace[0.379mm]}
\def\testa{\begin{table}[H]
\begin{tabular}{*{12}{c}}
\toprule
$r$&$\tg$&$g$&$b$&$p$&$\#$&$\tG$&$\tG$ \texttt{Id}&$G$ \texttt{Id}&\eqref{condB1}&$b\ge6$&\eqref{condB}\\ 
\midrule}
\def\coda{\bottomrule\end{tabular}\end{table}}
\def\interruzione{\coda \newpage \testa}

\testa

5 & 2 & 0 & 6 & 2 & 1 &$C_2^2$& $G(4,2)$ & $G(2, 1)$ & \phantom{X} & \checkmark & \checkmark \\ 
5 & 3 & 0 & 8 & 3 & 1 &$C_4$& $G(4,1)$ & $G(2, 1)$ & \checkmark & \checkmark & \checkmark \\ 
5 & 4 & 1 & 6 & 3 & 1 &$C_4$& $G(4,1)$ & $G(2, 1)$ & \phantom{X} & \checkmark & \checkmark \\ 
5 & 4 & 1 & 6 & 3 & 2 &$C_6$& $G(6,2)$ & $G(3, 1)$ & \checkmark & \checkmark & \checkmark \\ 
5 & 4 & 1 & 6 & 3 & 3 &$C_3\rtimes C_4$& $G(12,4)$ & $G(6, 1)$ & \phantom{X} & \checkmark & \checkmark \\ 
5 & 5 & 2 & 4 & 3 & 1 &$C_6$& $G(6,2)$ & $G(3, 1)$ & \checkmark & \phantom{X} & \checkmark \\ 
5 & 5 & 2 & 4 & 3 & 2 &$C_2\times C_4$& $G(8,2)$ & $G(4, 2)$ & \checkmark & \phantom{X} & \checkmark \\ 
5 & 5 & 1 & 8 & 4 & 1 &$C_2\times C_4$& $G(8,2)$ & $G(4, 2)$ & \phantom{X} & \checkmark & \checkmark \\ 
5 & 5 & 1 & 8 & 4 & 2 &$C_2\times C_4$& $G(8,2)$ & $G(4, 2)$ & \checkmark & \checkmark & \checkmark \\ 
5 & 5 & 1 & 8 & 4 & 3 &$D_8$& $G(16,7)$ & $G(8, 3)$ & \phantom{X} & \checkmark & \checkmark \\ 
5 & 5 & 1 & 8 & 4 & 4 &$C_2\times D_4$& $G(16,11)$ & $G(8, 5)$ & \phantom{X} & \checkmark & \checkmark \\ 
5 & 6 & 3 & 2 & 3 & 1,2 &$C_6$& $G(6,2)$ & $G(3, 1)$ & \checkmark & \phantom{X} & \checkmark \\ 
5 & 6 & 2 & 6 & 4 & 1,2,3 &$C_2\times C_6$& $G(12,5)$ & $G(6, 2)$ & \phantom{X} & \checkmark & \checkmark \\ 
5 & 7 & 4 & 0 & 3 & 1,2 &$C_2\times C_4$& $G(8,2)$ & $G(4, 1)$ & \checkmark & \phantom{X} & \checkmark \\ 
5 & 7 & 1 & 12 & 6 & 1 &$C_4\circ D_4$& $G(16,13)$ & $G(8, 5)$ & \phantom{X} & \checkmark & \checkmark \\ 
5 & 7 & 2 & 8 & 5 & 1 &$C_2\times C_4$& $G(8,2)$ & $G(4, 2)$ & \checkmark & \checkmark & \checkmark \\ 
5 & 8 & 4 & 2 & 4 & 2 &$C_2\times C_6$& $G(12, 5)$ & $G(6, 2)$ & \phantom{X} & \phantom{X} & \checkmark \\ 
5 & 9 & 5 & 0 & 4 & 1 &$C_2\times C_6$& $G(12,5)$ & $G(6, 2)$ & \checkmark & \phantom{X} & \checkmark \\ 
5 & 9 & 5 & 0 & 4 & 16 &$C_2^2\times C_4$& $G(16,10)$ & $G(8, 2)$ & \checkmark & \phantom{X} & \checkmark \\ 
5 & 9 & 3 & 8 & 6 & 1 &$C_8$& $G(8,1)$ & $G(4, 1)$ & \checkmark & \checkmark & \checkmark \\ 
5 & 9 & 3 & 8 & 6 & 2 &$C_8$& $G(8,1)$ & $G(4, 1)$ & \phantom{X} & \checkmark & \checkmark \\ 
5 & 9 & 3 & 8 & 6 & 3 &$C_2^2\times C_4$& $G(16,10)$ & $G(8, 5)$ & \checkmark & \checkmark & \checkmark \\ 
5 & 9 & 3 & 8 & 6 & 4 &$C_2^2\times C_4$& $G(16,10)$ & $G(8, 5)$ & \phantom{X} & \checkmark & \checkmark \\ 
5 & 9 & 3 & 8 & 6 & 5 &$C_4\circ D_4$& $G(16,13)$ & $G(8, 5)$ & \phantom{X} & \checkmark & \checkmark \\ 
5 & 10 & 5 & 2 & 5 & 1 &$C_2\times C_6$& $G(12, 5)$ & $G(6, 2)$ & \phantom{X} & \phantom{X} & \checkmark \\ 
5 & 10 & 5 & 2 & 5 & 2,3 &$C_2\times C_6$& $G(12,5)$ & $G(6, 2)$ & \checkmark & \phantom{X} & \checkmark \\ 
5 & 10 & 4 & 6 & 6 & 1 &$C_8$& $G(8,1)$ & $G(4, 1)$ & \checkmark & \checkmark & \checkmark \\ 
5 & 10 & 4 & 6 & 6 & 2 &$C_2\times C_6$& $G(12,5)$ & $G(6, 2)$ & \phantom{X} & \checkmark & \checkmark \\ 
5 & 10 & 4 & 6 & 6 & 3 &$Q_8\rtimes C_2$& $G(16,8)$ & $G(8, 3)$ & \phantom{X} & \checkmark & \checkmark \\ 
5 & 10 & 4 & 6 & 6 & 4 &$C_3\rtimes D_4$& $G(24,8)$ & $G(12, 4)$ & \phantom{X} & \checkmark & \checkmark \\ 
5 & 11 & 3 & 12 & 8 & 1 &$C_2\times Q_8$& $G(16,12)$ & $G(8, 5)$ & \phantom{X} & \checkmark & \checkmark \\ 
5 & 11 & 4 & 8 & 7 & 1 &$C_4\rtimes C_4$& $G(16,4)$ & $G(8, 3)$ & \phantom{X} & \checkmark & \checkmark \\ 
5 & 12 & 6 & 2 & 6 & 1,2 &$C_{10}$&$G(10,2)$ & $G(5, 1)$ & \checkmark & \phantom{X} & \checkmark \\ 
5 & 13 & 7 & 0 & 6 & 3 &$C_4^2$& $G(16, 2)$ & $G(8, 2)$ & \phantom{X} & \phantom{X} & \checkmark \\ 
5 & 13 & 7 & 0 & 6 & 4 &$C_4^2$&$G(16,2)$ & $G(8, 2)$ & \checkmark & \phantom{X} & \checkmark \\ 
5 & 13 & 5 & 8 & 8 & 1 &$C_2\times C_8$&$G(16,5)$ & $G(8, 2)$ & \phantom{X} & \checkmark & \checkmark \\ 
5 & 13 & 5 & 8 & 8 & 2 &$C_4\circ D_8$&$G(32,42)$ & $G(16, 11)$ & \phantom{X} & \checkmark & \checkmark \\ 
5 & 13 & 5 & 8 & 8 & 3,4 &$C_2\times C_4\circ D_4$&$G(32,48)$ & $G(16, 14)$ & \phantom{X} & \checkmark & \checkmark \\ 
5 & 15 & 5 & 12 & 10 & 1 &$C_6\rtimes C_4$&$G(24,7)$ & $G(12, 4)$ & \phantom{X} & \checkmark & \checkmark \\ 
5 & 15 & 7 & 4 & 8 & 2 &$C_{12}$&$G(12,2)$ & $G(6, 2)$ & \checkmark & \phantom{X} & \checkmark \\ 
5 & 16 & 8 & 2 & 8 & 1 &$C_2\times C_{10}$& $G(20, 5)$ & $G(10, 2)$ & \phantom{X} & \phantom{X} & \checkmark \\ 
5 & 21 & 9 & 8 & 12 & 1 &$C_8\circ D_4$&$G(32,38)$ & $G(16, 10)$ & \phantom{X} & \checkmark & \checkmark \\ 
5 & 29 & 13 & 8 & 16 & 1 &$Q_8\circ D_8$&$G(64,259)$ & $G(32, 46)$ & \phantom{X} & \checkmark & \checkmark \\ 

\interruzione

6 & 2 & 0 & 6 & 2 & 1 &$C_2$&$G(2,1)$ & $G(1, 1)$ & \phantom{X} & \checkmark & \checkmark \\ 
6 & 3 & 1 & 4 & 2 & 1 &$C_2^2$& $G(4, 2)$ & $G(2, 1)$ & \phantom{X} & \phantom{X} & \checkmark \\ 
6 & 5 & 3 & 0 & 2 & 1 &$C_2^3$& $G(8, 5)$ & $G(4, 2)$ & \phantom{X} & \phantom{X} & \checkmark \\ 
6 & 5 & 1 & 8 & 4 & 1 &$C_4$& $G(4,1)$ & $G(2, 1)$ & \checkmark & \checkmark & \checkmark \\ 
6 & 5 & 1 & 8 & 4 & 2 &$D_4$&$G(8,3)$ & $G(4, 2)$ & \phantom{X} & \checkmark & \checkmark \\ 
6 & 6 & 2 & 6 & 4 & 1 &$C_4$& $G(4,1)$ & $G(2, 1)$ & \phantom{X} & \checkmark & \checkmark \\ 
6 & 6 & 2 & 6 & 4 & 2 &$D_4$&$G(8,3)$ & $G(4, 2)$ & \phantom{X} & \checkmark & \checkmark \\ 
6 & 7 & 2 & 8 & 5 & 1 &$C_2\times C_4$& $G(8,2)$ & $G(4, 2)$ & \phantom{X} & \checkmark & \checkmark \\ 
6 & 8 & 4 & 2 & 4 & 1 &$C_6$& $G(6,2)$ & $G(3, 1)$ & \checkmark & \phantom{X} & \checkmark \\ 
6 & 9 & 5 & 0 & 4 & 3 &$C_2\times C_4$& $G(8,2)$ & $G(4, 1)$ & \checkmark & \phantom{X} & \checkmark \\ 
6 & 9 & 3 & 8 & 6 & 1 &$C_2\times C_4$& $G(8,2)$ & $G(4, 2)$ & \phantom{X} & \checkmark & \checkmark \\ 
6 & 10 & 4 & 6 & 6 & 1 &$C_6$& $G(6,2)$ & $G(3, 1)$ & \phantom{X} & \checkmark & \checkmark \\ 
6 & 13 & 7 & 0 & 6 & 1 &$C_2\times C_6$& $G(12, 5)$ & $G(6, 2)$ & \phantom{X} & \phantom{X} & \checkmark \\ 
6 & 13 & 5 & 8 & 8 & 1 &$C_4\circ D_4$& $G(16,13)$ & $G(8, 5)$ & \phantom{X} & \checkmark & \checkmark \\ 
6 & 14 & 6 & 6 & 8 & 1 &$C_8$& $G(8,1)$ & $G(4, 1)$ & \phantom{X} & \checkmark & \checkmark \\ 
6 & 17 & 5 & 16 & 12 & 1 &$D_4\circ Q_8$&$G(32,50)$ & $G(16, 14)$ & \phantom{X} & \checkmark & \checkmark \\ 
\midline
7 & 4 & 1 & 6 & 3 & 1,2 &$C_2^2$& $G(4,2)$ & $G(2, 1)$ & \phantom{X} & \checkmark & \checkmark \\ 
7 & 7 & 2 & 8 & 5 & 1 &$C_4$& $G(4,1)$ & $G(2, 1)$ & \checkmark & \checkmark & \checkmark \\ 
7 & 9 & 4 & 4 & 5 & 2 &$C_2\times C_4$& $G(8, 2)$ & $G(4, 1)$ & \phantom{X} & \phantom{X} & \checkmark \\ 
7 & 9 & 3 & 8 & 6 & 1 &$C_2\times C_4$& $G(8,2)$ & $G(4, 2)$ & \phantom{X} & \checkmark & \checkmark \\ 
7 & 10 & 5 & 2 & 5 & 1 &$C_6$& $G(6,2)$ & $G(3, 1)$ & \checkmark & \phantom{X} & \checkmark \\ 
7 & 11 & 6 & 0 & 5 & 1,2,3 &$C_2\times C_4$& $G(8, 2)$ & $G(4, 1)$ & \phantom{X} & \phantom{X} & \checkmark \\ 
7 & 13 & 7 & 0 & 6 & 2 &$C_2\times C_6$& $G(12,5)$ & $G(6, 2)$ & \phantom{X} & \phantom{X} & \checkmark \\ 
7 & 13 & 5 & 8 & 8 & 1 &$C_4\circ D_4$& $G(16,13)$ & $G(8, 5)$ & \phantom{X} & \checkmark & \checkmark \\ 
\midline
8 & 9 & 5 & 0 & 4 & 3--9 &$C_2^3$& $G(8, 5)$ & $G(4, 2)$ & \phantom{X} & \phantom{X} & \checkmark \\
8 & 9 & 3 & 8 & 6 & 1 &$C_4$& $G(4,1)$ & $G(2, 1)$ & \phantom{X} & \checkmark & \checkmark \\ 
8 & 13 & 7 & 0 & 6 & 1 &$C_2\times C_4$& $G(8, 2)$ & $G(4, 1)$ & \phantom{X} & \phantom{X} & \checkmark \\ 
\midline
9 & 6 & 2 & 6 & 4 & 1,2,3 &$C_2^2$& $G(4,2)$ & $G(2, 1)$ & \phantom{X} & \checkmark & \checkmark \\ 
9 & 11 & 6 & 0 & 5 & 2,3,4 &$C_2^3$& $G(8, 5)$ & $G(4, 2)$ & \phantom{X} & \phantom{X} & \checkmark \\ 

\coda




\begin{thebibliography} {99}
		
		
		\bibitem{acg2} E.~Arbarello, M.~Cornalba, and P.~A. Griffiths.
		\newblock {\em Geometry of algebraic curves. {V}ol. {II}}, volume
		268 of {\em Grundlehren der Mathematischen Wissenschaften}.
		\newblock Springer-Verlag, New York, 2011.
		
		\bibitem{bcv}
		Bardelli,  ~F., Ciliberto,  ~C., Verra,  ~A., Curves of minimal genus on a general abelian variety,  Compos. Math. 96 (1995), no. 2, 115--147. 
		
		
		
		
		
		
		
		
		\bibitem{birman-braids} J.~S. Birman.  \newblock {\em Braids, links,
			and mapping class groups}.  \newblock Princeton University Press,
		Princeton, N.J., 1974.  \newblock Annals of Mathematics Studies,
		No. 82.
		
		
		
		\bibitem{broughton-equi} S.~A. Broughton.  \newblock The equisymmetric
		stratification of the moduli space and the {K}rull dimension of
		mapping class groups.  \newblock {\em Topology Appl.},
		37(2):101--113, 1990.
		
		\bibitem{BL}
		C. Birkenhake, H. Lange, \emph{Complex abelian varieties.} Volume 302 of Grundlehren der Mathematischen Wissenschaften [Fundamental Principles of Mathematical Sciences]. Springer-Verlag, Berlin, second edition, 2004.
		
		\bibitem{baffo-linceo} F.~Catanese, M.~L{\"o}nne, and F.~Perroni.
		\newblock Irreducibility of the space of dihedral covers of the
		projective line of a given numerical type.  \newblock {\em Atti
			Accad. Naz. Lincei Cl. Sci. Fis. Mat. Natur. Rend. Lincei (9)
			Mat. Appl.}, 22(3):291--309, 2011.
		
		
		
		\bibitem{clp2} F.~Catanese, M.~L\"onne, and F.~Perroni.  \newblock
		Genus stabilization for the components of moduli spaces of curves
		with symmetries.  \newblock {\em Algebr. Geom.}, 3(1):23--49, 2016.
		
		
		
		
		
		
		
		\bibitem{cf1}Colombo, ~E., Frediani  ~P.,  A bound on the dimension of a totally geodesic submanifold in the Prym locus. Collectanea Mathematica,. 70, n.1, (2019), 51-57.  DOI: 10.1007/s13348-018-0215-0.  
		 \bibitem{cf2} E.~Colombo and P.~Frediani.  \newblock Second fundamental form of the Prym map in the ramified case. "Galois Covers, Grothendieck-Teichmueller Theory and Dessins d'Enfants - Interactions between Geometry, Topology, Number Theory and Algebra",  \newblock {\em Springer Proceedings in Mathematics $\&$ Statistics}, Vol.330, (2020). 
				
				
				
				\bibitem{cf3} E.~Colombo and P.~Frediani.  \newblock On the dimension of totally geodesic submanifolds in the Prym loci. 	arXiv:2101.05189. 		
		
		
		\bibitem{cfg} E.~Colombo, P.~Frediani, and A.~Ghigi.  \newblock On
		totally geodesic submanifolds in the {J}acobian locus.  \newblock{\em
			International Journal of Mathematics}, 26 (2015), no. 1, 1550005 (21 pages).
		
		
		\bibitem{cfgp} Colombo, ~E., Frediani, ~P., Ghigi, ~A., Penegini ~M., Shimura curves in the Prym locus.
		Communications in Contemporary Mathematics, Vol. 21, N0. 2 (2019) 1850009 (34 pages). DOI: 10.1142/S0219199718500098. 
		
		
		
		
		
		
		
		
		
		
		
		
		
		
		
		
		
		
		
		
		
		
		
		
		
		
		
		
		
		
		
		
		
		
		
		
		\bibitem{fgp} P.~Frediani, A.~Ghigi and M.~Penegini. \newblock Shimura
		varieties in the Torelli locus via Galois coverings. {\em
			Int. Math. Res. Not.} 2015, no. 20, 10595-10623.
		
		
		\bibitem{fgs} P.~Frediani, A.~Ghigi and I~Spelta. \newblock Infinitely many Shimura varieties in the Jacobian locus for $g \leq 4$.  arXiv:1910.13245. To appear in Ann. Scuola Norm. Sup. Pisa Cl. Sci. 
		
		\bibitem{fg} P.~Frediani, G.P.~Grosselli. \newblock Shimura curves in the Prym loci of ramified double covers.  arXiv:2007.09646.
		
		
		
		\bibitem{friedman-smith} Friedman, Robert; Smith, Roy, The generic
		Torelli theorem for the Prym map. {\em Invent. Math.} 67 (1982),
		no. 3, 473-490.
		
		
		
		
		
		\bibitem{gavino} G.~Gonz{\'a}lez~D{\'{\i}}ez and W.~J. Harvey.
		\newblock Moduli of {R}iemann surfaces with symmetry.  \newblock In
		{\em Discrete groups and geometry ({B}irmingham, 1991)}, volume 173
		of {\em London Math. Soc. Lecture Note Ser.}, pages
		75--93. Cambridge Univ. Press, Cambridge, 1992.
		
		
		
		
		
		
		
		
		
		\bibitem{ikeda} A. Ikeda,  \newblock Global Prym-Torelli Theorem for double coverings of elliptic curves,  {\em Algebr. Geom.} 7 (2020), no. 5, 544-560.
		
		\bibitem{kanev-global-Torelli} V.~I. Kanev.  \newblock A global
		{T}orelli theorem for {P}rym varieties at a general point.
		\newblock {\em Izv. Akad. Nauk SSSR Ser. Mat.}, 46(2):244--268, 431,
		1982.
		
		
		\bibitem{lange-ortega} H.~Lange and A.~Ortega.  \newblock Prym varieties of
		cyclic coverings.  \newblock {\em Geom. Dedicata}, 150:391--403,
		2011.
		
		\bibitem{magma} MAGMA Database of Small Groups; http://magma.maths.usyd.edu.au/magma/htmlhelp/
		text404.htm.script. 
		
		
		
		
		
		
		
		\bibitem{pietro-vale} V.O.~Marcucci, G.P. ~Pirola.  \newblock Generic
		Torelli theorem for Prym varieties of ramified
		coverings. Compos. Math. 148 (2012), no. 4, 1147--1170.
		
		
		\bibitem{mn} Marcucci,  ~V, Naranjo,  ~J.C., Prym varieties of double coverings of elliptic curves,  Int. Math. Res. Notices 6 (2014), 1689-1698.
		
		\bibitem{moh} Mohajer,  ~A,  On Shimura subvarieties of the Prym locus.  arXiv:1804.10131v2. To appear in Communications in Algebra.
		
		\bibitem{MZ}
		A. Mohajer, K. Zuo, \emph{On Shimura subvarieties generated by families of abelian covers of $\mathbb{P}^{1}$.} Journal of Pure and Applied Algebra, 222 (4), 2018, 931-949.
		
		
		\bibitem{moonen-linearity-1} B.~Moonen.  \newblock Linearity
		properties of {S}himura varieties. {I}.  \newblock {\em J. Algebraic
			Geom.}, 7(3):539--567, 1998.
		
		
		
		\bibitem{moonen-oort} B.~Moonen and F.~Oort.  \newblock The {T}orelli
		locus and special subvarieties.  \newblock In {\em {H}andbook of
			{M}{oduli: Volume II}}, pages 549--94.  International {P}ress,
		Boston, MA, 2013.
		
		\bibitem{M10}
		B. Moonen, \emph{Special subavarieties arising from families of cyclic covers of the projective line}. Documenta Mathematica 15 (2010) 793-819.
		
		\bibitem{mumford-Shimura} D.~Mumford.  \newblock A note of {S}himura's
		paper ``{D}iscontinuous groups and abelian varieties''.  \newblock  {\em Math. Ann.}, 181:345--351, 1969.
		
		
		
		
		\bibitem{nagarama} D. S., Nagaraj; S. Ramanan, S.  \newblock
		Polarisations of type (1,2,...,2) on abelian varieties.   \newblock
		{\em Duke
			Math. J.} 80 (1995), no. 1, 157-194.
		
		
		
		
		
		\bibitem{no} Naranjo, ~J.C., Ortega, ~ A., Verra, ~A.,  {\em Trans. Amer. Math. Soc.} 371 (2019), 3627-3646.

\bibitem{no1}
J.C.~Naranjo, A.~Ortega, Global Prym-Torelli for double coverings ramified in at least $6$ points,   arxiv:2005.11108. To appear in {\em Journal of Algebraic Geometry}. 
		
		
		
		
		
		
		
		
		
		
		\bibitem{P2}
		R. Pardini, \emph{On the period map for abelian covers of projective varieties.} Annali della Scuola Normale Superiore di Pisa - Classe di Scienze 26.4 (1998): 719-735.
		
		\bibitem{penegini2013surfaces} M.~Penegini.  \newblock Surfaces
		isogenous to a product of curves, braid groups and mapping class
		groups.  \newblock In {\em Beauville surfaces and groups}, volume
		123 of {\em Springer Proc. Math. Stat.}, pages 129--148. Springer,
		Cham, 2015.
		
		
		
		
		
		
		
		\bibitem{W}
		A. Wright, \emph{Shwarz triangle mappings and Teichm\"uller curves: abelian square-tiled surfaces},  J. Mod. Dyn. 6 (2012) 405- 426.
	
		
		
	\end{thebibliography}
\end{document}